\documentclass{article}
\usepackage{graphicx} 
\usepackage{color}
\usepackage[colorlinks]{hyperref}
\newcommand{\red}[1]{\textcolor{red}{#1}} 
\newcommand*{\QEDA}{\hfill\hbox{\vrule width1.0ex height1.0ex}}
\usepackage{mdframed}
\usepackage[most]{tcolorbox}

\newtcolorbox{eqbox}{
  colback=gray!5!white,
  colframe=gray!80!black,
  boxrule=0.5pt,
  arc=4pt,
  left=6pt,
  right=6pt,
  top=4pt,
  bottom=4pt,
  enhanced
}

\usepackage{amsmath}
\usepackage{amsfonts}
\usepackage{latexsym}
\usepackage{amssymb}

\newtheorem{thm}{Theorem}[section]
\newtheorem{theorem}[thm]{Theorem}

\newtheorem{lemma}[thm]{Lemma}

\newtheorem{proposition}[thm]{Proposition}

\newtheorem{remark}[thm]{Remark}

\newcommand{\beq}{\begin{equation}}
\newcommand{\eeq}{\end{equation}}
\newcommand{\beqa}{\begin{eqnarray}}
\newcommand{\eeqa}{\end{eqnarray}}
\newcommand{\beqas}{\begin{eqnarray*}}
\newcommand{\eeqas}{\end{eqnarray*}}
\newcommand{\bi}{\begin{itemize}}
\newcommand{\ei}{\end{itemize}}

\newcommand{\vgap}{\vspace{.1in}}

\newcommand{\nn}{\nonumber}

\newcommand{\R}{\mathbb{R}}

\newcommand{\lam}{{\lambda}}

\newcommand{\inner}[2]{\langle #1,#2\rangle}
\newcommand{\Inner}[2]{\left \langle #1\,,#2 \right \rangle}

\newcommand{\dom}{\mathrm{dom}\,}

\newcommand{\Argmin}{\mathrm{Argmin}\,}
\newcommand{\argmin}{\mathrm{argmin}\,}

\newcommand{\bConv}[1]{\overline{\mbox{\rm Conv}}\,(\R^{#1})}

\newcommand{\mConv}[1]{\overline{\mbox{\rm Conv}}_\mu\,(\R^{#1})}
\newcommand{\nConv}[1]{\overline{\mbox{\rm Conv}}_\nu\,(\R^{#1})}

\newcommand{\BB}{\operatorname{BB}}

\title{Universal subgradient and proximal bundle methods for \\ convex and strongly convex hybrid composite optimization}

\author{
		Vincent Guigues \thanks{School of Applied Mathematics FGV/EMAp, 22250-900 Rio de Janeiro, Brazil. (email: {\tt vincent.guigues@fgv.br}).} 
        \qquad
		Jiaming Liang \thanks{
        Goergen Institute for Data Science and Artificial Intelligence (GIDS-AI) and Department of Computer Science, University of Rochester, Rochester, NY 14620 (email: {\tt jiaming.liang@rochester.edu}). This work was partially supported by GIDS-AI seed funding and AFOSR grant FA9550-25-1-0182.}
        \qquad
            Renato D.C. Monteiro \thanks{School of Industrial and Systems
			Engineering, Georgia Institute of
			Technology, Atlanta, GA 30332.
			(email: {\tt renato.monteiro@isye.gatech.edu}). This work
			was partially supported by AFOSR Grants FA9550-22-1-0088 and FA9550-25-1-0131.}
	}

\date{July 14, 2024 (revisions: August 2, 2024; June 25, 2025; November 20, 2025)}

\begin{document}

\maketitle

\begin{abstract}
This paper 
develops two parameter-free methods for
solving
convex and strongly convex hybrid composite universal
problems, namely,
 a composite subgradient
 type method and a proximal bundle type method.
Functional complexity bounds for the two methods are established in terms of the unknown strong convexity parameter. The two proposed methods are  universal with
respect to all problem parameters, including the strong convexity one,
and 
require no knowledge of the optimal value.
Moreover, in contrast to previous works,
they do not
restart nor use multiple threads.\\

		\\
		
  
\end{abstract}



\section{Introduction}\label{sec:intro}






This paper considers convex hybrid composite optimization (HCO) problem
\begin{equation}\label{eq:ProbIntro2}
	\phi_{*}:=\min \left\{\phi(x):=f(x)+h(x): x \in \R^n\right\},
	\end{equation}
where $f, h: \R^{n} \rightarrow \R\cup \{ +\infty \} $ are proper lower semi-continuous  convex functions such that
	$ \dom h \subseteq \dom f $ and the following conditions hold:
 there exist scalars $M_f\ge 0$ and $L_f \ge 0$ and a first-order oracle $f':\dom h \to \R^n$ (i.e., $f'(x)\in \partial f(x)$ for every $x \in \dom h$) satisfying the $(M_f,L_f)$-hybrid condition that
$\|f'(x)-f'(y)\| \le 2M_f + L_f \|x-y\|$ for every $x,y \in \dom h$.

The intrinsic convexity parameter
$\mu_\psi$
of a convex function $\psi$ is defined as the largest scalar $\mu$ such that $\psi(\cdot)-\mu\|\cdot\|^2/2$ is convex.
A method for solving \eqref{eq:ProbIntro2} is called parameter-free if it does not require knowledge of any parameter associated with the instance $(f,h)$, such as
   $(M_f,L_f)$, or the intrinsic convexity parameters $\mu_f$, $\mu_h$, and  $\mu_\phi$.
Parameter-free methods whose complexities are expressed in terms of
$(M_f,L_f)$ are called
universal methods.
Moreover,
parameter-free methods whose complexities are expressed in terms of
$(M_f,L_f, \mu_\phi)$, $(M_f,L_f, \mu_f)$, and $(M_f,L_f, \mu_h)$,
are referred to as
$\mu_\phi$-universal, $\mu_f$-universal, and $\mu_h$-universal,   respectively. 
It is worth noting that
$\mu_\phi$ can be substantially larger than
$\mu_f + \mu_h$
(e.g., for $\alpha\gg 0$, $f(x)=\alpha \exp(x)$, and $h(x)=\alpha \exp(-x)$, we have $\mu_\phi =2\alpha \gg 0 = \mu_f + \mu_h$).

The main goal of this paper is to develop  methods for solving \eqref{eq:ProbIntro2} 
  that: 
  \begin{itemize}
      \item[(1)]
  are $\mu_\phi$-universal for any instance $(f,h)$ of \eqref{eq:ProbIntro2} with arbitrary parameter pair $(M_f,L_f)$;
  \item[(2)] 
  for any given positive scalars $\bar M$ and $\bar \mu$, are optimal for the class of instances of \eqref{eq:ProbIntro2}
such that $L_f=0$,
$M_f \le \bar M$ and
$\mu_\phi \ge \bar \mu$.
\end{itemize}
It is worth noting that a $\mu_h$-universal
(resp., $\mu_f$-universal)   method that is optimal for
the class of instances of \eqref{eq:ProbIntro2}
such that $L_f=0$,
$M_f \le \bar M$ and
$\mu_h \ge \bar \mu$
(resp., $\mu_f \ge \bar \mu$)
is not necessarily optimal for the above class, as the first one (i.e., with $\mu_\phi \ge \bar \mu$) is larger than the latter one (i.e., with $\mu_h \ge \bar \mu$ or $\mu_f \ge \bar \mu$).

{\bf Related literature.}
We divide our discussion here into universal and $\mu$-universal methods.

{\it Universal methods:}
The first universal   methods for solving \eqref{eq:ProbIntro2} under the condition that $\nabla f$ is Hölder continuous 
have been presented in
\cite{nesterov2015universal}
and 
\cite{lan2015bundle}.
Specifically, the first paper 
develops universal   variants of the primal and dual gradient methods, and an optimal universal   accelerated gradient method, while the second one 
develops accelerated universal   variants of the bundle-level and the prox-level methods which achieve optimal complexity.
Additional universal   methods
for solving \eqref{eq:ProbIntro2} have been studied in \cite{liang2021average,liang2023average,mishchenko2020adaptive,zhou2024adabb} under the condition that $f$ is smooth,
and  in \cite{li2023simple,liang2020proximal,liang2023unified} for the case where $f$ is either smooth or nonsmooth.
The methods in \cite{liang2020proximal,liang2023unified} (resp., \cite{mishchenko2020adaptive,zhou2024adabb}) are also shown to be $\mu_h$-universal  
(resp., $\mu_f$-universal   under the condition that $h=0$).
The papers \cite{liang2021average,liang2023average} present
universal   accelerated composite gradients
methods for solving \eqref{eq:ProbIntro2} under the more general condition that $f$
is a smooth $m$-weakly convex function. Since any
convex function is $m$-weakly convex for any $m>0$, the results of \cite{liang2021average,liang2023average} also apply to the convex case and yield complexity bounds similar to the ones of \cite{mishchenko2020adaptive,zhou2024adabb}.

{\it $\mu_\phi$-Universal methods:}
Under the assumption that $f$ is a smooth function (i.e., $M_f=0$),
various works \cite{alamo2019gradient,alamo2022restart,alamo2019restart,aujol2023parameter,aujol2023fista,aujol2022fista,fercoq2019adaptive,lan2023optimal,nesterov2013gradient,renegar2022simple} have developed $\mu_\phi$-universal (or $\mu_f$-universal) methods   for \eqref{eq:ProbIntro2} 
with  $\tilde {\cal O}(\sqrt{L_f/\mu_\phi})$ (or $\tilde {\cal O}(\sqrt{L_f/\mu_f})$) iteration complexity bound.
For the sake of our discussion, we refer to a convex (resp.,
strongly convex) version of an accelerated gradient method as  ACG
(resp., S-ACG).
Among papers concerned with
finding an $\varepsilon$-solution of \eqref{eq:ProbIntro2}, 
\cite{nesterov2013gradient}
proposes the first $\mu_f$-universal   method
based on a restart S-ACG scheme
where each iteration adaptively updates an estimate of $\mu_f$
and calls S-FISTA
(see also \cite{fercoq2019adaptive} for a $\mu_\phi$-universal   variant along this venue);
moreover,
under the assumption that $\phi_*$ and $L_f$ are known, \cite{aujol2023fista}  
develops a $\mu_\phi$-universal   method that performs only one call to an ACG variant (for convex CO).

Finally, motivated by previous works such as \cite{ioudnest14,lan2013iteration, linxiao14,nemnest85,nesterov2013gradient,o2015adaptive,aspremontall2017},
$\mu_\phi$-universal   restart schemes, including accelerated ones,
were developed in
\cite{renegar2022simple} (see also related paper \cite{grimmer2023optimal})
for both when $\phi$ is smooth or nonsmooth.
Specifically, 
\cite{renegar2022simple}
develops
a $\mu_\phi$-universal   method under the assumption that $\phi_*$ is known; and,
also an alternative one for when $\phi_*$
is unknown\footnote{For this algorithm, it is assumed that the number of threads is ${\cal O}(\log \varepsilon^{-1})$ and, it is  shown that 
its complexity is better than the usual one by a logarithmic factor. 
However, if  a single thread is used,
Corollary 5 of \cite{renegar2022simple} implies that the  complexity for its nonsmooth approach (without smoothing) applied to
strongly convex instances of \eqref{eq:ProbIntro2}
reduces to the usual convex complexity (i.e., with $\mu_{\phi}=0$), and hence is not $\mu_{\phi}$-universal
  (see the last paragraph of Subsection \ref{subsec:U-CS} for more details about this method and its complexity).}, which solves the original problem
for multiple initial points and tolerances
in parallel.

{\bf Our contribution.}
We present
two  $\mu_\phi$-universal   methods for solving \eqref{eq:ProbIntro2} that fulfill the two requirements (1) and (2) above, namely:
a composite subgradient (U-CS) type method and a
proximal bundle (U-PB) type method.
In contrast to
\cite{grimmer2023optimal,renegar2022simple},
they are 
non-restart
methods
that do not require $\phi_*$ to be known.
Moreover,  
 U-CS and U-PB do not need the use of multiple threads
as  the parallel version of the restart subgradient
universal method of \cite{renegar2022simple} does.

U-CS is a variant of the universal   primal gradient method of \cite{nesterov2015universal}
(see also Appendix C.2 of \cite{liang2023unified}),
which  is  not known to be $\mu_\phi$-universal.
U-PB is a variant of the generic proximal bundle (GPB) method of \cite{liang2023unified} that bounds the number of consecutive null iterations
and adaptively chooses the prox stepsize under this policy.
Both methods are analyzed in a unified manner using a general framework for strongly convex optimization problems \eqref{eq:ProbIntro2} (referred to as FSCO)
which specifies minimal
conditions for its instances to be $\mu_\phi$-universal.
A functional complexity bound in terms of $\mu_\phi$ is established for any FSCO instance 
which is then used to derive complexities for both U-CS and U-PB.

{\bf Organization of the paper.}
    Subsection~\ref{subsec:DefNot}  presents basic definitions and notation used throughout the paper.
Section~\ref{sec:CS pf} presents two $\mu_\phi$-universal methods, namely U-CS and U-PB, for solving problem \eqref{eq:ProbIntro2} and establishes their corresponding complexity bounds.
Section~\ref{sec:FSCO} formally
describes FSCO, which is the framework that include U-CS and U-PB as special instances,
and provides the oracle complexity of FSCO.
Section \ref{sec:pf1} is devoted to the proofs of the main results for U-CS and U-PB.
Section~\ref{sec:conclusion} presents some concluding remarks and possible extensions.
Finally, Appendix~\ref{sec:appendix} provides technical results of FSCO and U-PB.

\subsection{Basic definitions and notation} \label{subsec:DefNot}
    
    Let $\R$ denote the set of real numbers.
    Let $ \R_+ $ (resp., $ \R_{++} $) denote the set of non-negative real numbers (resp., the set of positive real numbers).
	Let $\R^n$ denote the standard $n$-dimensional Euclidean space equipped with  inner product and norm denoted by $\left\langle \cdot,\cdot\right\rangle $
	and $\|\cdot\|$, respectively. 
	Let $\log(\cdot)$ denote the natural logarithm.
	
	For given $\Phi: \R^n\rightarrow (-\infty,+\infty]$, let $\dom \Phi:=\{x \in \R^n: \Phi (x) <\infty\}$ denote the effective domain of $\Phi$ 
	and $\Phi$ is proper if $\dom \Phi \ne \emptyset$.
	A proper function $\Phi: \R^n\rightarrow (-\infty,+\infty]$ is $\mu$-convex for some $\mu \ge 0$ if
	\[
	\Phi(\alpha x+(1-\alpha) y)\leq \alpha \Phi(x)+(1-\alpha)\Phi(y) - \frac{\alpha(1-\alpha) \mu}{2}\|x-y\|^2
	\]
	for every $x, y \in \dom \Phi$ and $\alpha \in [0,1]$.
Let $\mConv{n}$ denote the set of all proper lower semicontinuous $\mu$-convex functions.
 We simply denote $\mConv{n}$ by $\bConv{n}$ when $\mu=0$.
	For $\varepsilon \ge 0$, the \emph{$\varepsilon$-subdifferential} of $ \Phi $ at $x \in \dom \Phi$ is denoted by
	\[
        \partial_\varepsilon \Phi (x):=\left\{ s \in\R^n: \Phi(y)\geq \Phi(x)+\left\langle s,y-x\right\rangle -\varepsilon, \forall y\in\R^n\right\}.
        \]
	We denote the subdifferential of $\Phi$ at $x \in \dom \Phi$ by $\partial \Phi (x)$, which is the set  $\partial_0 \Phi(x)$ by definition.
    For a given subgradient
$\Phi'(x) \in \partial \Phi(x)$, we denote the linearization of convex function $\Phi$ at $x$ by $\ell_\Phi(\cdot,x)$, which is defined as
\begin{equation}\label{linfdef}
\ell_\Phi(\cdot,x)=\Phi(x)+\langle \Phi'(x),\cdot-x\rangle.
\end{equation}



\section{Universal composite subgradient and proximal bundle methods}\label{sec:CS pf}

This section presents two completely universal methods for solving \eqref{eq:ProbIntro2}, namely, U-CS in Subsection \ref{subsec:U-CS} and U-PB in Subsection \ref{sec:bundle}. Their iteration-complexity guarantees are described in this section but their proofs are postponed to Section \ref{sec:pf1},
after the presentation of
a general framework for (not necessarily composite) convex optimization problems in Section \ref{sec:FSCO}.
Viewing U-CS and U-PB as special cases of this framework will enable us to prove the main complexity results of this section in a unified manner in Section \ref{sec:pf1}.





We assume that 
\begin{itemize}
\item[(A1)] $f, h \in \bConv{n}$ are such that
		$\dom h \subset \dom f$, and a subgradient oracle,
		 i.e.,
		a function $f':\dom h \to \R^n$
		satisfying $f'(x) \in \partial f(x)$ for every $x \in \dom h$, is available;
		\item[(A2)]
		there exists $(M_f,L_f)\in \R^2_+$ such that for every $x,y \in \dom h$,
		\[
		\|f'(x)-f'(y)\| \le 2M_f + L_f \|x-y\|;
		\]
\item[(A3)]
		the set of optimal solutions $X_*$ of
		problem \eqref{eq:ProbIntro2} is nonempty.
\end{itemize}

It is known that that (A2) implies that for every $x,y \in \dom h$, 
\begin{equation}\label{ineq:est} 	
f(x)-\ell_f(x;y) \le 2M_f \|x-y\| + \frac{L_f}{2} \| x-y \|^2.
	\end{equation}
When $M_f=0$, (A2) implies that $f$ is $L_f$-smooth, and when $L_f=0$, (A2) implies that $f$ is $M_f$-Lipschitz continuous. Hence, our analysis applies to both smooth and nonsmooth, as well as hybrid, instances of \eqref{eq:ProbIntro2}, where both $L_f$ and $M_f$ are positive.

Recall that for a given function $\phi \in \bConv{n}$, the  intrinsic convex parameter
$\mu_\phi$
of $\phi$ is defined as the largest scalar $\mu$ such that $\phi(\cdot)-\mu\|\cdot\|^2/2$ is convex.
Clearly, $\mu_\phi \ge 0$ for any $\phi \in \bConv{n}$.
Even though, for the sake of completeness, the complexity results developed in this paper apply to any $\phi \in \bConv{n}$, our main interest in this paper is to develop $\mu_\phi$-universal methods for instances of \eqref{eq:ProbIntro2} such that
\begin{eqbox}
\begin{equation*}
\mu_\phi >0.
\end{equation*}
\end{eqbox}

Our goal in this section is to describe two completely universal methods for solving \eqref{eq:ProbIntro2}  and describe  their iteration-complexity guarantees for finding 
a $\bar \varepsilon$-solution $\bar x$ of \eqref{eq:ProbIntro2}, i.e., 
a point $\bar  x$ such that $\phi(\bar x) - \phi_* \le \bar \varepsilon$, in terms of
$(M_f,L_f,\mu_\phi )$, $\bar \varepsilon$, and $d_0$, where $d_0$
is the distance of the initial iterate $\hat x_0$ with respect to $X_*$, i.e.,
\begin{equation}\label{defd0}
d_0 := \min \{\|x-\hat x_0\| : x \in X_*\}.
\end{equation}

For any given positive scalars $\bar M$ and $\bar \mu$, we consider two classes of instances $(f,h)$ of problem \eqref{eq:ProbIntro2}, namely: 1) $L_f=0$, $M_f \le \bar M$, and $\mu_\phi \ge \bar \mu$; and 2) $L_f=0$, $M_f \le \bar M$ and $\mu_h \ge \bar \mu$. The following remarks about them hold:
i) the second 
class is a subset of the first one;
and, ii) 
the lower bound $\Omega(\bar M^2/(\bar \mu \bar \varepsilon))$ has been established for the second class in \cite{liang2020proximal} using an instance where $f$ is piece-wise linear and $h(x)=\bar \mu\|x\|^2/2$ (see (73) and (74) of \cite{liang2020proximal-arxiv}, which is an extended version of \cite{liang2020proximal}). 
These two observations imply  that $\Omega(\bar M^2/(\bar \mu \bar \varepsilon))$ is also a
lower bound for the first class.
Both the adaptive subgradient method and the proximal bundle method studied in Appendix C.2 and Section 3 of \cite{liang2023unified}, respectively, are shown to have the following properties:
i) they are  $\mu_h$-universal for any instance $(f,h)$ of \eqref{eq:ProbIntro2} with arbitrary $(M_f,L_f)$; and
ii) they achieve near optimal complexity $\tilde {\cal O}(\bar M^2/(\bar \mu \bar \varepsilon))$ for the second  class mentioned above.

\subsection{A universal composite subgradient method}\label{subsec:U-CS}

This subsection presents the first of two aforementioned universal methods for solving \eqref{eq:ProbIntro2}.
Specifically, it states a variant of the universal primal gradient method of \cite{nesterov2015universal} for solving \eqref{eq:ProbIntro2}, referred to as the U-CS method,
and establishes its complete universality.
Moreover, it is shown that its $\mu_\phi$-universality
does not require
 any line search on the intrinsic convex parameter $\mu_\phi$.

We start by stating U-CS.

	
	
	

\noindent\rule[0.5ex]{1\columnwidth}{1pt}
	
	U-CS($\hat x_0,\chi,\lam_0,\bar \varepsilon$)
	
	\noindent\rule[0.5ex]{1\columnwidth}{1pt}
	
    {\bf Input:} $(\hat x_0,\chi,\lam_0,\bar \varepsilon)\in \dom h \times [0,1) \times \R_{++} \times \R_{++}$.
        \begin{itemize}
		\item [0.]  Set $\lam=\lam_0$ and $k=1$;
	    \item[1.] Compute 
	    \begin{equation}\label{def:x}
	        x =\underset{u\in  \R^n}\argmin
		\left\lbrace  \ell_f(u;\hat x_{k-1})+h(u) +\frac{1}{2\lam} \|u- \hat x_{k-1} \|^2 \right\rbrace;
	    \end{equation}
        if $\phi(x) - \phi_* \le \bar \varepsilon$, then {\bf stop};
		\item[2.]  {\bf If} the inequality $$f(x)-\ell_f(x;\hat x_{k-1}) - \frac{1-\chi}{2\lam} \|x-\hat x_{k-1}\|^2 \le \frac{(1-\chi) \bar \varepsilon}2$$ does not hold, {\bf then} set $\lam = \lam/2$ and go to step~1; {\bf else}, 
        set $\lam_k=\lam$, $\hat x_k=x$,  $k\leftarrow k+1$, and go to step 1.
	\end{itemize}
	\rule[0.5ex]{1\columnwidth}{1pt}

We now make some remarks about U-CS. 
First, U-CS does not require any knowledge about the parameter pair $(M_f,L_f)$ or the intrinsic convex parameter $\mu_\phi$.
Second, the U-CS method with $\chi=0$ is exactly the universal primal gradient method
analyzed in \cite{nesterov2015universal}, which  establishes its universality  with
respect to the parameters of a Hölder condition on $f$ but not with respect to $\mu_h$.
The subsequent paper \cite{liang2023unified} establishes the $\mu_h$-universality  
of the U-CS method with $\chi=0$ (see Proposition C.2 of \cite{liang2023unified}), but left
its
$\mu_\phi$-universality as an open question.
Third, Theorem \ref{thoutercs} below shows that U-CS, with $\chi \in (0,1)$ and $\chi$ not too close to $1$, is $\mu_h$-universal,
and also  $\mu_\phi$-universal if in addition  $\chi$ is not close to $0$.
The latter extra condition on $\chi$ is because its 
iteration complexity bound, in terms of $\mu_\phi$ only, depends on
 $\chi^{-1}$, and hence does not apply
to the universal primal gradient method of \cite{nesterov2015universal} (i.e.,  U-CS with $\chi=0$).

We now make some comments about the universality of the U-CS method.
First, it performs a line search on $\lam$ on step 2 due to the fact that
$(M_f,L_f)$ is not assumed to be known.
Lemma \ref{lem:lam} below  shows that when  $(M_f,L_f)$ is known then a constant prox stepsize, which provably satisfies the inequality in step 2 for every $k$,
can be used throughout the method. Thus, the line search on step 2 is only performed so as to make the U-CS method $(M_f,L_f)$-parameter-free. 
Second, the $\mu$-universality of the U-CS method will be established in Theorem~\ref{thoutercs} 
regardless of whether
the line search is performed (i.e., when $(M_f,L_f)$ is not known) or is not needed (i.e., when $(M_f,L_f)$ is known).

\begin{lemma}\label{lem:lam}
If $\lam$ satisfies 
\begin{equation}\label{ineq:lam}
    \lam \le \frac{(1-\chi)^2 \bar \varepsilon}{ 4{M}_f^2 + \bar \varepsilon {L}_f},
\end{equation}
then the point $x=x_k(\lam)$ computed in \eqref{def:x} satisfies the inequality in step 2. As a consequence, if the initial prox stepsize $\lam_0$ is chosen to be the right hand side of \eqref{ineq:lam}, then $\lam$ remains constant throughout the method.
\end{lemma}

\begin{proof}
    Assume that $\lam$ satisfies \eqref{ineq:lam}. Using \eqref{ineq:est} with $(M_f, L_f, u,v)=(M_f, L_f, x,\hat x_{k-1})$ and the inequality 
    $2M_f \Delta - a \Delta^2 \leq M_f^2/a$ with
    $\Delta=\|x-\hat x_{k-1}\|$ and $a=(1-\chi-\lam L_f)/2\lam$, we have
    \begin{align*}
    f(x)-\ell_f(x;\hat x_{k-1}) +\frac{\chi-1}{2 \lambda}  \|x-\hat x_{k-1}\|^2
    &\stackrel{\eqref{ineq:est}}\le 2M_f \|x-\hat x_{k-1}\| + \frac{L_f}2 \|x-\hat x_{k-1}\|^2+\frac{\chi-1}{2 \lambda  }\|x-\hat x_{k-1}\|^2 \\
    &= 2M_f \|x-\hat x_{k-1}\| - \frac{1-\chi-\lam L_f}{2\lam} \|x-\hat x_{k-1}\|^2 \\
    &\le \frac{2\lam M_f^2}{1-\chi-\lam L_f} \stackrel{\eqref{ineq:lam}}\le \frac{(1-\chi)\bar \varepsilon}2, 
    \end{align*}
    where the last inequality is due to
    \eqref{ineq:lam}.
    Hence, the lemma is proved.
\end{proof}

 The main result of this subsection,
whose proof is postponed until Subsection \ref{subsec:proof-CS}, is stated next.
Its goal is to establish the complete universality of U-CS.



\begin{theorem}\label{thoutercs}
U-CS terminates in at most
\begin{equation}\label{complprimal-2}
\min\left\{\min\left[ \frac{1}{\chi}\left(1+\frac{Q({\bar \varepsilon})}{\bar \varepsilon \mu_{\phi}}\right),  1+\frac{Q({\bar \varepsilon})}{\bar \varepsilon \mu_h}\right]
\log \left(1+\frac{\mu_\phi d_0^2}{\bar \varepsilon}\right)\, , \, \frac{d_0^2 Q({\bar \varepsilon})}{{\bar \varepsilon}^2} \right\}+
\displaystyle \left \lceil 
2 \log \frac{\lambda_0 Q({\bar \varepsilon})}{\bar \varepsilon} \right \rceil
\end{equation}
iterations and
\begin{equation}\label{defQf}
Q({\bar \varepsilon})= \frac{8 M_f^2}{(1-\chi)^2}+ {\bar \varepsilon} \left( \lam_0^{-1} + 
\frac{2L_f}{(1-\chi)^2} \right).
\end{equation}
\end{theorem}

We now make some remarks about Theorem \ref{thoutercs}.
First,
if 
$\lambda_0^{-1}={\cal O}(L_f)$ and
$\chi \in [0,1)$ is such that 
$(1-\chi)^{-1} = {\cal O}(1)$, then
$Q(\bar{\varepsilon})/{\bar \varepsilon} 
=\tilde {\cal O}\left(M_f^2/\bar \varepsilon + L_f\right)$, 
and hence 
\eqref{complprimal-2} in terms of $\mu_h$ is  
\begin{equation}\label{compU-PB-a}
    \tilde {\cal O}\left(\frac{M_f^2}{{\bar \varepsilon} \mu_h} + \frac{L_f}{\mu_h}
\right),
\end{equation}
which is the same well-known bound established for the U-CS method with $\chi=0$ (e.g., see Proposition C.2 of \cite{liang2023unified}).

Second, if in addition to the assumptions on the pair $(\chi,\lam_0)$ made in the first remark,
the parameter $\chi$ is
also assumed to be
not close to zero
 (i.e., $ \chi^{-1} = {\cal O}(1)$), then the complexity bound \eqref{complprimal-2} is also 
\begin{equation}\label{compU-PB}
    \tilde {\cal O} \left( \frac{Q(\bar{\varepsilon})}{\chi \bar \varepsilon \mu_\phi}\right)
=\tilde {\cal O}\left(\frac{M_f^2}{ {\bar \varepsilon} \mu_\phi} + \frac{L_f}{\mu_\phi} \right),
\end{equation}
which is considerably smaller than \eqref{compU-PB-a} when $ \mu_\phi \gg \mu_h$.

Third, in view of the lower complexity remarks made in the paragraph following \eqref{defd0}, bound \eqref{compU-PB} with $(\mu_\phi,M_f)$ replaced by $(\bar \mu,\bar M)$ is optimal, up to logarithmic terms, for the  class of instances $(f,h)$ of problem \eqref{eq:ProbIntro2} where $L_f=0$, $M_f \le \bar M$, and $\mu_\phi \ge \bar \mu$.

Fourth, algorithms with optimal complexity for
the class of instances $(f,h)$ such that
$L_f\le \bar L$, $M_f \le \bar M$, and $\mu_\phi \ge \bar \mu$
have been studied in \cite{lan2015bundle,nesterov2015universal} for the case where $\bar \mu=0$, and in \cite{alamo2019gradient,alamo2022restart,alamo2019restart,aujol2023parameter,aujol2023fista,aujol2022fista,fercoq2019adaptive,lan2023optimal,nesterov2013gradient,renegar2022simple} for the case where
$\bar M=0$ and $\bar \mu>0$.


Finally, we end this subsection with a comparison of the complexity bound of Theorem \ref{thoutercs} with the
one derived in Corollary 5 of \cite{renegar2022simple} for its subgradient variant under the assumption that $f$ has quadratic growth. 
Specifically,
under the assumption that $M_f$ is known and $L_f=0$, Corollary 5 of \cite{renegar2022simple} shows that their  subgradient variant has (parallel) iteration complexity bound
\begin{equation}\label{bound:Ben}
    {\cal O}\left(\frac{M_f^2}{\bar\varepsilon \mu_f} + N + \left(\frac{M_f d_0}{2^N \bar \varepsilon}\right)^2\right)
\end{equation}
where $N$ is the number of threads. As a consequence, they conclude that the above bound reduces to
\[
{\cal O}\left(\frac{M_f^2}{\bar\varepsilon  \mu_f } + \log \frac1{\bar \varepsilon} + M_f^2d_0^2\right)
\]
when $N={\cal O}(\log \bar \varepsilon^{-1})$.
For sufficiently small $\bar \varepsilon>0$, the above bound and Theorem 3.2.1 of \cite{nesterov2018lectures} show that their subgradient variant is exactly optimal for the class of instances $(f,h)$ of problem \eqref{eq:ProbIntro2} where $L_f=0$, $M_f \le \bar M$, and $\mu_f \ge \bar \mu$.
However,  bound \eqref{bound:Ben} with $N=1$ reduces to
${\cal O}(M_f^2 d_0^2/\bar \varepsilon^2)$, and hence does not imply that their single-thread subgradient variant
is optimal
for the aforementioned class.

\if{
Complexity for stationary conditions applies to U-CS using U-CS parameters. We obtain the following theorem
on complexity of stationarity conditions for U-CS:

\begin{theorem}\label{thm:main1cs} 
Let $\bar x_k$, $\bar s_k$, $\bar \varepsilon_k$
given by \eqref{xkcs} and \eqref{def:sk-ek-12bcs}.
For a given tolerance pair $(\hat \varepsilon,\hat \rho) \in \R^2_{++}$, U-CS with 
\[
    \chi \in (0,1), \quad \varepsilon=\frac{\chi (1-\chi)\hat \varepsilon}{14},
\]
generates a triple $(\bar x_k,\bar s_k,\bar \varepsilon_k)$ satisfying 
\[
\bar s_k \in \partial \phi_{\bar \varepsilon_k}(\bar x_k), \quad \|\bar s_k\| \leq \hat \rho, \quad \bar \varepsilon_k \le \hat \varepsilon
\]
in at most
\begin{equation}\label{bound:total-1}
\min\left\{\frac{(1+\nu \underline{\lam})(1+\mu \underline{\lam})}{\underline{\lam}[\nu(1+ \mu \underline{\lam})+\chi(\mu-\nu)]} \log \left(1+ \sigma \beta(\hat \varepsilon,\hat \rho)
\right) \, , \, \beta(\hat \varepsilon,\hat \rho)\right\}+\displaystyle \left \lceil 
2 \log\left(\frac{2 \lambda_0 T_{\varepsilon}^2}{(1-\chi)\varepsilon} \right) \right \rceil
\end{equation}
iterations, where
\[
\beta(\hat \varepsilon,\hat \rho)=\frac{4d_0}{{\underline \lambda} \hat \rho} + \frac{4\chi \hat \varepsilon}{7{\underline \lambda} \hat \rho^2} + \frac{2(2+\chi)d_0^2}{\chi {\underline \lambda} \hat \varepsilon},
         \quad \sigma = \frac{\underline{\lam}[\nu(1+\mu \underline{\lam})+\chi(\mu-\nu)]}{1+\mu {\underline \lambda}+\chi {\underline \lambda}(\nu-\mu)}
\]    
where $\underline \lambda$ is given in \eqref{ineq:lamj2}.
\end{theorem}

}\fi

\subsection{A universal proximal bundle method}\label{sec:bundle}

This subsection describes the U-PB method and  states a result that ensures its complete universality. 

Conditions (A1)-(A3) are assumed to hold in this subsection. We also assume 
\begin{itemize}
\item[(A4)] the diameter of the domain of $h$ given by $\max\{\|x-y\|:x,y \in \dom h\}$ is bounded by 
$D<+\infty$.
\end{itemize}


The U-PB method is an extension of the GPB method of \cite{liang2023unified}.
In contrast to GPB, we use an adaptive stepsize and introduce a maximal number $\overline N$ (which can be as small as 2) of iterations for all cycles.
Similarly to U-CS, U-PB is another instance of FSCO and we establish functional complexity for U-PB using the results of Section \ref{sec:FSCO}.
Compared with the complexity results in \cite{liang2023unified}, those obtained in this paper are sharper, since they are expressed in terms of $\mu_\phi$ instead of $\mu_h$.


U-PB is based on the following bundle update (BU) blackbox which builds a model
$m_f^{+}+h$
for $f+h$ on the basis of a previous model $m_f$ of $f$ and of a new linearization
$\ell_{f}(\cdot,x)$
of $f$. This blackbox BU$(x^c,x, m_f,\lambda)$ is given below
and takes as inputs a prox-center $x^c$,
a current approximate solution $x$, 
an initial model $m_f$ for $f$, and a stepsize $\lambda>0$.

    \noindent\rule[0.5ex]{1\columnwidth}{1pt}
	
	BU$(x^c,x, m_f,\lambda)$

    \noindent\rule[0.5ex]{1\columnwidth}{1pt}
    {\bf Inputs:} $\lam \in \R_{++}$ and
    $(x^c,x,m_f(\cdot)) \in \R^n \times
    \R^n \times {\overline{\mbox{Conv}}}(\mathbb{R}^n)$ such that $m_f(\cdot) \le f(\cdot)$ and
$$
x = \underset{u\in \R^n}\argmin \left\{ m_f(u)+h(u) + \frac{1}{2\lam} \|u-x^c\|^2 \right\}.
$$    
 Find function $m_f^{+}(\cdot) \in  {\overline{\mbox{Conv}}}(\mathbb{R}^n)$ such that
\begin{equation}\label{def:Gamma}
        \max\{{\overline m_f(\cdot)}, \ell_f(\cdot;x)\} \le m_f^{+}(\cdot) \le f(\cdot),
	\end{equation}
   where $\ell_f(\cdot;\cdot)$ is as in \eqref{linfdef} and
   $\overline m_f(\cdot) $ is such that
\begin{equation}\label{def:bar Gamma} 
\overline m_f(\cdot) \leq f(\cdot), \quad
 \overline m_f(x) = m_f(x), \quad 
	x = \underset{u\in \R^n}\argmin \left \{\overline m_f(u) + h(u)+\frac{1}{2\lam} \|u-x^c\|^2 \right\}.\\
	\end{equation}
{\bf Output:} $m_f^+$.\\     
    \noindent\rule[0.5ex]{1\columnwidth}{1pt}

In the following, we give two examples of BU, namely two-cuts and multiple-cuts schemes. The proofs for the two schemes belonging to BU can be provided similarly to Appendix D of \cite{liang2023unified}.
    
    \begin{itemize} 
        \item [(E1)] \textbf{two-cuts scheme:}
        We assume that $m_f$ is of the form $m_f=\max \{A_f,\ell_f(\cdot;x^-)\}$
        where $A_f$ is an affine function satisfying $A_f\le f$. 
        The scheme then sets
        $A_f^+(\cdot) :=  \theta A_f(\cdot) + (1-\theta) \ell_f(\cdot;x^-)$
     and updates $m_f^{+}$ as $m_f^{+}(\cdot)  := \max\{A_f^+(\cdot),\ell_f(\cdot;x)\}$,
     where $\theta \in [0,1]$ satisfies
        \begin{align*}
            &\frac1\lam (x-x^c) + \partial h(x) 
+ \theta \nabla A_f + (1-\theta) f'(x^-) \ni 0, \\
            &\theta A_f(x) + (1-\theta) \ell_f(x;x^-) = \max \{A_f(x),\ell_f(x;x^-)\}.
        \end{align*}
     
	    \item [(E2)] \textbf{multiple-cuts scheme:} We assume that $m_f$ has the form
	    $m_f=m_f(\cdot;B)$ where
	    $B \subset \R^n$ is the current bundle set and $m_f(\cdot;B)$ is defined as 
     $m_f(\cdot;B) := \max \{ \ell_f(\cdot;b) : b \in B \}$.
	     This scheme selects the next bundle set $B^+$ so that   
        $B(x) \cup \{x\} \subset B^+ \subset B \cup \{x\}$
        where $ B(x) := \{ b \in B : \ell_f(x;b) = m_f(x) \}$,
         and then outputs $m_f^+ = m_f(\cdot;B^+)$.
	\end{itemize}

Before giving the motivation of U-PB, we briefly review the GPB method of \cite{liang2023unified}.
GPB is an inexact proximal point method (PPM, with fixed stepsize) in that,
given a prox-center $\hat x_{k-1} \in \R^n$ and a prox stepsize $\lam >0$, it computes the next prox-center $\hat x_k$ as a suitable approximate solution of 
the prox subproblem
\begin{equation}\label{eq:key-prox}
   \hat x_k \approx 
  \underset{u\in \R^n}\argmin \left\{ (f+h)(u)+\frac1{2\lam}\|u-\hat x_{k-1}\|^2 \right\}.
\end{equation}
More specifically,  a sequence of prox bundle subproblems of the form
\begin{equation}\label{def:xj}
	    x_{j} =\underset{u\in \R^n}\argmin
	    \left\lbrace  (f_j+h)(u) +\frac{1}{2\lam}\|u- \hat x_{k-1} \|^2 \right\rbrace,  
	    \end{equation}
where 
$f_j \le f$ is a bundle
approximation of $f$,
is solved
until for the first time an iterate $x_j$ as in \eqref{def:xj} approximately solves
\eqref{eq:key-prox}, and such $x_j$ is then set to be $\hat x_k$.
The bundle approximation $f_j$
is sequentially updated, for example, according to either one of the schemes (E1) and (E2) described above.

 U-PB is also an inexact PPM
but with variable prox stepsizes
(i.e., with $\lam$ in \eqref{eq:key-prox} replaced by $\lam_k$)
instead of a constant one as in GPB.
Given iteration upper limit $\overline N \ge 1$ and prox-center $\hat x_{k-1}$,
it adaptively computes $\lam_k>0$ as follows:
starting with $\lam=\lam_{k-1}$, it solves at most $\overline N$ prox subproblems of the form \eqref{def:xj} in an attempt to obtain an approximate solution of \eqref{eq:key-prox} and, if it fails,
repeats this procedure with $\lam$ divided by $2$;
otherwise, it sets $\lam_k$ to be the first successful $\lam$ and $\hat x_{k}$ as described in the previous paragraph.



We are now ready to state U-PB.


\noindent\rule[0.5ex]{1\columnwidth}{1pt}
	
	U-PB($\hat x_0,\chi,\lam_0,\bar \varepsilon,\overline N$)
	
	\noindent\rule[0.5ex]{1\columnwidth}{1pt}
      \par {\textbf{Input:}} $(\hat x_0, \lambda_0, \chi, \bar \varepsilon, \overline N)  \in 
      \dom h \times \mathbb{R}_{++} \times [0,1) \times \mathbb{R}_{++} \times \mathbb{N}_{++}$.
	\begin{itemize}
		\item [0.] Set $\lambda=\lambda_0$, 
	$y_0=\hat x_0$, $N=0$, $j=1$, and
  $k =1$. Find
  $f_1 \in \overline{\mbox{Conv}}(\mathbb{R}^n)$ such that 
  $\ell_f(\cdot;{\hat x}_0) \le f_1(\cdot) \leq f(\cdot)$;
	    \item[1.]
	    Compute $x_j$ as in \eqref{def:xj}; if $\phi(x_j)-\phi_* \leq \bar \varepsilon$ stop;
            \item[2.] Compute

\begin{equation}\label{def:txj2}
         \bar \phi_j = \left\{ \begin{array}{ll} \displaystyle \phi(x_j) +
		\frac{\chi}{2\lam} \|x_j-\hat x_{k-1}\|^2, & \mbox{  if } N=0,  \\
        \displaystyle  \min\left\{ \bar \phi_{j-1}\,,\,\phi(x_j) +
		\frac{\chi}{2\lam} \|x_j-\hat x_{k-1}\|^2   \right\}, & \mbox{ otherwise,}
         \end{array}
         \right.
         \end{equation}

 set  $N=N+1$ and
		\begin{equation}\label{ineq:hpe1}
		    t_{j} = \bar \phi_j - \left((f_j+h)(x_{j}) +\frac{1}{2\lam}\|x_{j} - \hat x_{k-1} \|^2\right);
		\end{equation}

		\item[3.] {\textbf{If}} $t_j>(1-\chi){\bar \varepsilon}/2$ and $N < \overline N$ {\textbf{then}}\\ 
         3.a.$\hspace*{0.8cm}$perform a {\textbf{null update}, i.e.:}
         set  $f_{j+1}=\mbox{BU}(\hat x_{k-1},x_j,f_j,\lambda)$;\\
        {\textbf{else}}\\
          $\hspace*{0.8cm}${\textbf{if}} $t_j >(1-\chi){\bar \varepsilon}/2$ and $N = \overline N$\\
    3.b.$\hspace*{1.2cm}$perform a {\textbf{reset update}, i.e.,}
   set $\lambda \leftarrow \lambda/2$;\\
   $\hspace*{0.8cm}${\textbf{else}} (i.e., $t_j \le (1-\chi){\bar \varepsilon}/2$ and $N \le \overline N$) \\
         	 3.c.$\hspace*{1.2cm}$perform a {\bf serious update}, i.e., set $\hat x_k = x_{j}$, $\hat \Gamma_k=f_j+h$, $\hat y_k=y_j$, $\lambda_k=\lambda$;\\
             $\hspace*{1.7cm}$$k \leftarrow k+1$;\\
          $\hspace*{0.8cm}${\textbf{end if}}\\ 
     3.d.$\hspace*{0.2cm}$set $N=0$ and find $f_{j+1}\in \overline{\mbox{Conv}}(\mathbb{R}^n)$ such that
			$\ell_f(\cdot;{\hat x}_{k-1}) \le f_{j+1}\le f$;

	{\textbf{end if}}\\
         	\item[4.] Set  $j\leftarrow j+1$ and go to step 1.
	\end{itemize}
	\rule[0.5ex]{1\columnwidth}{1pt}

\if{

\noindent\rule[0.5ex]{1\columnwidth}{1pt}
	
	U-PB
	
	\noindent\rule[0.5ex]{1\columnwidth}{1pt}
	\begin{itemize}
		\item [0.] Let $ \hat x_0\in \dom h $, $\lambda_1=\lambda>0$, $\chi \in [0,1)$, $\varepsilon>0$, and integer $\overline N \geq 1$  be given, and set  
	$y_0=\hat x_0$, $N=0$, $j=1$, and
  $k =1$. Find
  $f_1 \in \overline{\mbox{Conv}}(\mathbb{R}^n)$ such that 
  $\ell_f(\cdot;{\hat x}_0) \le f_1 \leq f$;
	    \item[1.]
	    Compute $x_j$ as in \eqref{def:xj};
   
         \item[2.] Choose
		$y_{j} \in \{ x_{j}, y_{j-1}\}$ such that
		\begin{equation}\label{def:txj}
	\phi(y_j) +
		\frac{\chi}{2\lam} \|y_j-\hat x_{k-1}\|^2 
  = \min \left \lbrace \phi(x_j) +
		\frac{\chi}{2\lam} \|x_j-\hat x_{k-1}\|^2, \phi(y_{j-1}) +
		\frac{\chi}{2\lam} \|y_{j-1}-\hat x_{k-1}\|^2 \right\rbrace,
		\end{equation}
		and set  $N=N+1$ and
		\begin{equation}\label{ineq:hpe1}
		    t_{j} = \phi(y_j) +
		\frac{\chi}{2\lam} \|y_j-\hat x_{k-1}\|^2 - \left((f_j+h)(x_{j}) +\frac{1}{2\lam}\|x_{j} - \hat x_{k-1} \|^2\right);
		\end{equation}

		\item[3.] {\textbf{If}} $t_j>\varepsilon$ and $N < \overline N$ {\textbf{then}}\\  
         $\hspace*{0.8cm}$perform a {\textbf{null update}, i.e.:}
         set  $f_{j+1}=\mbox{BU}(\hat x_{k-1},x_j,f_j,\lambda)$;\\
        {\textbf{else}}\\
          $\hspace*{0.8cm}${\textbf{if}} $t_j\leq \varepsilon$\\
     	 $\hspace*{1.2cm}$perform a {\bf serious update}, i.e., set $\hat x_k = x_{j}$, $\hat \Gamma_k=f_j+h$, $\hat y_k=y_j$, $\lambda_k=\lambda$, and go  to Step 5;\\
   $\hspace*{0.8cm}${\textbf{else}} {\bf  (i.e., $t_j > \varepsilon$ and $N = \overline N$)} \\
              $\hspace*{1.2cm}$perform a {\textbf{reset update}, i.e.,}
   set $\lambda \leftarrow \lambda/2$;\\
          $\hspace*{0.8cm}${\textbf{end if}}\\ 
     $\hspace*{0.8cm}$set $N=0$ and find $f_{j+1}\in \overline{\mbox{Conv}}(\mathbb{R}^n)$ such that
			$\ell_f(\cdot;{\hat x}_{k-1}) \le f_{j+1}\le f$;

	{\textbf{end if}}\\
         	\item[4.] Set  $j\leftarrow j+1$ and go to step 1.
	\end{itemize}
	\rule[0.5ex]{1\columnwidth}{1pt}
}\fi

We now give further explanation about U-PB.
U-PB performs three types of iterations, namely, null, reset, and 
serious, corresponding to 
the types of updates performed at the end.
A reset (resp., serious) cycle of U-PB consists of a reset
(resp., serious) iteration and all the consecutive null iterations preceding it.
The index $j$ counts the total number of iterations including null, reset, and serious ones.
The index $k$ counts the serious cycles which, together with the quantities $\hat x_k$, $\hat y_k$, and $\hat \Gamma_k$ computed at the end of cycle $k$, is used to cast U-PB as an instance of FSCO. All iterations within a cycle are referred to as inner iterations. The quantity $N$ counts the number of inner iterations performed in the current cycle.
Each cycle of U-PB performs at most $\overline N$ iterations.
A serious cycle successfully finds $t_j\le (1-\chi)\bar \varepsilon/2$ within $\overline N$ iterations, while a reset cycle fails to do so. In both cases, U-PB resets the counter $N$ to $0$ and starts a new cycle.
The differences between the two cases are: 1) the stepsize $\lam$ is halved at the end of a reset cycle, while it is kept as is at the end of a serious cycle; and 2) the prox-center is kept the same at the end of a reset cycle, but it is updated to the latest $x_j$ at the end of a serious cycle. 
\if{
We now make some remarks about the first iteration of
a cycle of U-PB. First,  if
$t_i \le \varepsilon$,
then $x_i$ satisfies 
the inequality
\begin{align*}
  t_{i} &= \phi(y_i) +
		\frac{\chi}{2\lam} \|y_i-\hat x_{k-1}\|^2 - \left((f_i+h)(x_{i}) +\frac{1}{2\lam}\|x_{i} - \hat x_{k-1} \|^2\right) \\
       &= \phi(y_i) +
		\frac{\chi}{2\lam} \|y_i-\hat x_{k-1}\|^2 - \left(\ell_f(x_i;\hat x_{k-1})+h(x_{i}) +\frac{1}{2\lam}\|x_{i} - \hat x_{k-1} \|^2\right) \\ 
    &\le \phi(x_i) +
		\frac{\chi}{2\lam} \|x_i-\hat x_{k-1}\|^2 - \left(\ell_f(x_i;\hat x_{k-1})+h(x_{i}) +\frac{1}{2\lam}\|x_{i} - \hat x_{k-1} \|^2\right) \\     
        &= 
    f(x_i) - \ell_f(x_i;\hat x_{k-1}) - (1-\chi) \frac{\|x_i-\hat x_{k-1}\|^2}{2\lam}   
\end{align*}
}\fi

We now make three remarks about U-PB.
First, U-CS is a special case of U-PB with $\overline N=1$ and 
$f_{j+1}=\ell_f(\cdot;{\hat x}_{k-1})$ in step 3.d.
Second, 
it follows from the fact that $f_j\le f$ and the definition of $t_j$ in \eqref{ineq:hpe1} that the primal gap of the prox subproblem in \eqref{eq:key-prox} at the iterate $y_j$ is upper bounded by $t_j+(1-\chi)\|y_j-\hat x_{k-1}\|^2/(2\lam)$. Hence, if $t_j\le (1-\chi){\bar \varepsilon}/2$, then $y_j$ is an $\varepsilon_j$-solution of \eqref{eq:key-prox} where $\varepsilon_j=(1-\chi) [ {\bar \varepsilon}/2+\|y_j-\hat x_{k-1}\|^2/(2\lam)]$.
Third,
the iterate $y_{j}$ computed in step 2 of U-PB  satisfies
\begin{equation}\label{eq:yj-intuition}
    y_j \in \Argmin \left\lbrace \phi(x) + \frac{\chi}{2\lam}\|x-\hat x_{k-1}\|^2 :
		x \in \{x_{i},\ldots,x_j\}
		\right\rbrace,
\end{equation}
where $i$ denotes the first iteration index of the cycle containing $j$.
In other words, $y_j$ is the best point 
in terms of $\phi(\cdot) + \chi\|\cdot-\hat x_{k-1}\|^2/(2\lam)$
among all the  points obtained
in the course of solving \eqref{def:xj} and the point
$\hat y_{k-1}$ obtained at the end
of the previous cycle.



We now state the first main result of this subsection where  the functional  iteration complexity of U-PB
is established.

\begin{theorem}\label{complprimal-2U-PBt}
Assuming that $\dom h$ has a finite diameter $D>0$.
Let 
\begin{equation}\label{defa}
B=8+ 12 \log\left(1+\frac{L_f^2 D^2 \overline N}{16 M_f^2}\right)
\end{equation}
and define $U(\bar \varepsilon)$ by
\begin{equation}\label{defU}
U(\bar \varepsilon):=
\frac{4(M_f^2 + \bar \varepsilon L_f) B }{(1-\chi)^2} + \frac{ \overline N \bar \varepsilon}{\lam_0}.
 \end{equation}
The total number of inner iterations (i.e., the ones indexed by $j$) performed by U-PB
is bounded by
\begin{equation}\label{complprimal-2ag}
 \min\left\{\min\left[ \frac{1}{\chi}\left({\overline N}+  \frac{U({\bar \varepsilon})}{\bar \varepsilon \mu_\phi}\right),  {\overline  N}+\frac{U({\bar \varepsilon})}{\bar \varepsilon \mu_h}\right]
\log \left(1 + \frac{\mu_\phi d_0^2}{\bar \varepsilon}\right) \, , \, \frac{d_0^2 U({\bar \varepsilon})}{\bar \varepsilon^2} \right\}
+
{\overline  N} \displaystyle \left \lceil 
2 \log  \frac{\lambda_0 
U(\bar \varepsilon)}{{\overline N}\bar \varepsilon} \right \rceil.
\end{equation}
\end{theorem}

The complexity result of Theorem
\ref{complprimal-2U-PBt}, under the mild conditions that
$\chi$ is neither close
to one nor zero (i.e., $\max\{ \chi^{-1},(1-\chi)^{-1}\} = {\cal O}(1)$),
\[
\overline  N = {\cal O} \left(\frac{M_f^2 + \bar \varepsilon L_f }{ \bar \varepsilon \mu_\phi } \min \left\{
1 , \lam_0 \mu_\phi  \right\} \right),
\]
and $\mu_\phi \gg \mu_h$,
reduces 
to 
\begin{equation}\label{compU-PB-1}
\tilde {\cal O}\left( \frac{M_f^2}{\bar \varepsilon \mu_\phi} + \frac{L_f}{\mu_\phi} \right),
\end{equation}
which is similar to the 
functional complexity bound
\eqref{compU-PB}
obtained for U-CS.
Moreover, similar 
to U-CS,
bound \eqref{compU-PB-1} with $(\mu_\phi,M_f)$ replaced by $(\bar \mu,\bar M)$ is nearly optimal for the  class of instances $(f,h)$ of problem \eqref{eq:ProbIntro2} where $L_f=0$, $M_f \le \bar M$, and $\mu_\phi \ge \bar \mu$, because of the lower complexity remarks made in the paragraph following \eqref{defd0}.

\section{A functional framework for strongly convex optimization}\label{sec:FSCO}

This section presents a general framework, namely FSCO, for (strongly) convex optimization problems and establishes an oracle-complexity bound for it.

FSCO is presented in
the context of
the convex optimization problem
\begin{equation}\label{eq:ProbIntro}
	\phi_{*}:=\min \left\{\phi(x): x \in \R^n\right\}
	\end{equation}
where $\phi \in \bConv{n}$.
It will be shown in this section that any instance of FSCO is $\mu_\phi$-universal.
Moreover, its oracle-complexity bound will be used in Subsections~\ref{subsec:proof-CS} and \ref{sec:proofupb} to establish the iteration-complexities of two important instances, namely, the U-CS method 
described in Subsection 
\ref{subsec:U-CS}, and the U-PB method described in Subsection 
\ref{sec:bundle}. It is worth mentioning, though, that 
FSCO is only used to establish universality relative to $\mu_\phi$.
Complete universality of the two specific instances mentioned above, and hence relative to the parameter pair $(M_f,L_f)$, is established by
performing line searches on the prox stepsize
$\lam_k$, as described in step 2 of U-CS and in step 3 of U-PB.

Given a prox-center $x^- \in \dom \phi$, every iteration of FSCO approximately solves a prox subproblem of the form
\begin{equation}\label{def:prox}
    \min_{u\in \R^n} \left\{\phi(u)+\frac{1}{2 \lam}\|u-x^-\|^2\right\},
\end{equation}
for some $\lam>0$, to obtain an inexact solution $y$
whose primal gap size is used to terminate FSCO.
The black-box below, which is repeatedly invoked by FSCO, describes how the above subproblem is approximately solved. 




\noindent\rule[0.5ex]{1\columnwidth}{1pt}
	
	 Black-box (BB)
     $(x^-,\chi,\varepsilon)$
	
	\noindent\rule[0.5ex]{1\columnwidth}{1pt}
    {\bf Input:}  $(x^-,\chi,\varepsilon) \in \dom \phi \times [0,1) \times \R_{++}$.
    
    \vgap
    
    \noindent
    {\bf Output:}
    $(x,y,\Gamma,\lam) \in \dom \phi \times \dom \phi \times \bConv{n} \times \R_{++}$ satisfying $\Gamma \le \phi$,
\begin{equation}\label{propyb}
    \phi(y)+\frac{\chi}{2 \lambda}\|y-x^-\|^2 - \underset{u\in\R^n}\min\left\{\Gamma(u) + \frac{1}{2\lam}\|u- x^- \|^2\right\} \le \varepsilon,
\end{equation}
and
\beq 
    x = \underset{u\in\R^n} \argmin \left\lbrace \Gamma(u) + \frac{1}{2\lam}\|u- x^- \|^2 \right\rbrace. \label{defxhatk}
     \eeq
    
\noindent\rule[0.5ex]{1\columnwidth}{1pt}


We now make some remarks about BB.
First, $x^-$ is the current prox center and $(\chi,\bar \varepsilon)$ are tolerances used to quantify the inexactness of $y$ as a solution of \eqref{def:prox}. Second, \eqref{propyb} quantifies the inexactness of $y$ in terms of a functional certificate
$\Gamma$,
which is a bundle for $\phi$.
More precisely, $y$ is a $\hat \varepsilon$-solution of \eqref{def:prox}, where $\hat \varepsilon = \varepsilon + (1-\chi)\|y- x^- \|^2/(2\lam)$. Indeed, this follows from \eqref{propyb} and the fact that $\Gamma\le \phi$.
Finally, even though the next prox-center $x$ is viewed as an output
of BB, it
is uniquely determined
by $\Gamma$
as in  \eqref{defxhatk}, and hence could alternatively be removed from the BB output. In other words, the relevant output produced by BB is $(\lam, y,\Gamma)$.


The following result describes a key contraction inequality that holds for the two prox-centers $x^-$ and $x$ of BB.
Its proof is postponed to Subsection \ref{subsec:contraction} and uses a key result stated in Appendix \ref{subsec:tech-FSCO}.

\begin{proposition}\label{propouter}
Assume that  $\phi \in \bConv{n}$ and let $\mu=\mu_\phi$.
Assume also that
$(x,y,\Gamma,\lam) = \BB(x^-,\chi,\varepsilon)$
for some $(x^-,\chi,\varepsilon) \in \dom \phi \times [0,1) \times \R^n_{++}$, and $\Gamma \in \nConv{n}$ for some $\nu \in [0,\mu]$.
Then, for every $u \in \R^n$, we have
    
\begin{equation}\label{ineq:contraction}
    2 \lambda\left[\phi(y)-\phi(u)\right]  \le \frac{2\lam \varepsilon}{1-\chi} + \left\|x^- - u\right\|^2 - (1+\sigma) \|x-u\|^2,
\end{equation}
where 
\begin{equation}\label{eq:sigma}
    \sigma = \sigma(\mu,\nu) := \frac{\lam[\nu(1+\mu \lam)+\chi(\mu-\nu)]}{1+\mu  \lambda+\chi  \lambda(\nu-\mu)}.
\end{equation} 
\end{proposition}

We now motivate the assumption made
by Proposition \ref{propouter} on $\Gamma$ by
considering the special case \eqref{eq:ProbIntro2} of problem \eqref{eq:ProbIntro}
where $\phi=f+h$ and $f$ and $h$ are
as in (A1)-(A3).
In this setting, many algorithms (see e.g., U-CS and U-PB in Section \ref{sec:CS pf}) generate  $\Gamma$ functions of the form
$\Gamma_f+h$ where $\Gamma_f$ is a convex (possibly, affine) function underneath $f$.
Clearly, letting $\nu =\mu_h$, it follows that bundles $\Gamma$ generated in this form are $\nu$-convex, since so is $h$. In conclusion, the $\nu$-convexity of $\Gamma$ is inherited from the $\nu$-convexity assumption on the composite component $h$
of $\phi$, i.e., the part of $\phi$ that is not approximated by simpler functions in the prox subproblems generated while solving \eqref{eq:ProbIntro2}.


We make some remarks about Proposition \ref{propouter}.
If $\sigma$ in \eqref{eq:sigma} is such that $\sigma=\Theta(\lam \mu)$, then
it will be shown in Theorem \ref{thouter} that an algorithm that repeatedly calls BB finds an $\bar \varepsilon$-solution of \eqref{eq:ProbIntro} in
 $ {\cal O}((\lam \mu)^{-1} \log (\bar \varepsilon^{-1}))$ calls.
 We mention two cases where $\sigma = \Theta(\lam \mu)$. The first case is
 when $\nu \approx \mu$ and $\chi$ is arbitrary, but this case does not hold for many instances of \eqref{eq:ProbIntro}. In particular, the ``contraction" formula \eqref{ineq:contraction} for the pair $(\nu,\chi)=(\mu,0)$
 is well-known (e.g., see
 Lemma 4.1(a) of \cite{liang2023unified}).
The second case is when
 $\chi^{-1} = {\cal O}(1)$ and $\nu \in [0,\mu]$.
 Indeed, it follows from the first inequality in Lemma \ref{lem:sigma}(a) that $\sigma=\Theta(\lam \mu)$ when  $\chi^{-1} = {\cal O}(1)$.
 Finally, even though $\chi$ chosen very close to one satisfies the condition in the second case, such choice of $\chi$ is not desirable as it makes the term $2\lam \varepsilon/(1-\chi)$ in \eqref{ineq:contraction} too large.
 Thus, in view of the previous remarks, a good choice of $\chi$ is one such that
 $\max\{\chi^{-1}, (1-\chi)^{-1}\} = {\cal O}(1)$.

 We now describe FSCO.

 \noindent\rule[0.5ex]{1\columnwidth}{1pt}
	
	 FSCO$(\hat x_0, \chi, \bar \varepsilon)$
	
	\noindent\rule[0.5ex]{1\columnwidth}{1pt}
	\begin{itemize}
		\item [0.] Let $(\hat x_0, \chi, \bar \varepsilon) \in \dom \phi \times [0,1) \times \R_+ $ be given and set $k=1$;

     
    \item[1.] Compute 
$(\hat x_k, \hat y_k, \hat \Gamma_k, \lambda_k) = \BB(\hat x_{k-1},\chi, (1-\chi) \bar \varepsilon/2)$;
    
		\item[2.]
  Check whether $\phi(\hat y_k) - \phi_* \le \bar \varepsilon$ and if so {\bf stop}; else set $k\leftarrow k+1$ and go to step 1.
	\end{itemize}
	\rule[0.5ex]{1\columnwidth}{1pt}
 
	
	



We now make some comments about FSCO.

First, FSCO generates two sequences, i.e., $\{\hat x_k\}$ and
$\{\hat y_k\}$.
If FSCO stops, then the last iterate  of the second one is guaranteed to be a $\bar \varepsilon$-solution of \eqref{eq:ProbIntro} due to step 2 of FSCO. Hence, the second sequence ensures that FSCO stops, while the 
 first one should be viewed as the sequence of prox-centers generated by the bundle sequence $\{\hat \Gamma_k\}$.

Second,
since BB does not provide a specific recipe for computing $(\hat x_k, \hat y_k, \hat \Gamma_k, \lambda_k)$ in step 1,
FSCO is not a completely specified algorithm but rather a framework built with the sole purpose
of studying $\mu_\phi$-universality of its instances.
Specifically,
FSCO provides minimal conditions on its generated sequence $(\hat x_k, \hat y_k, \hat \Gamma_k, \lambda_k)$ (i.e., only that is generated by iteratively calling BB) to ensure $\mu_\phi$-universality of its instances.

Finally, for some instances of FSCO, such as U-CS and U-PB,
$\lam_k$ is computed by performing a line search on $\lam$ within
specific implementations of BB, because some parameters associated with $\phi$
(e.g., $L_f$ and $M_f$ in the setting of Section \ref{sec:CS pf}) are assumed to be unknown.
When these parameters are known, the line search can be avoided by choosing $\lam_k$ in terms of these parameters.


 The main complexity result for FSCO is relative to the  instances that satisfy the following condition:
\begin{itemize}
\item[(F1)] there exists $\underline \lambda>0$ such that
$\lambda_k \geq {\underline \lambda}$ for every $k \ge 0$.
\end{itemize}


\begin{theorem}\label{thouter}
The number of iterations performed by any instance of  FSCO that satisfies condition (F1) is bounded by
\begin{equation}\label{complprimal}
 \min\left\{\min\left[ \frac{1}{\chi}\left(1+\frac{1}{\underline{\lam} \mu}\right),  1+\frac{1}{\underline{\lam} \nu}\right] \log\left(1+ \frac{\mu d_0^2}{\bar \varepsilon}\right), \frac{d_0^2}{{\underline \lambda}{\bar \varepsilon}} \right\},
\end{equation}
where the scalars $\chi$ and $\bar \varepsilon$ are input to FSCO, $d_0$ is as in \eqref{defd0}, and
\begin{equation}\label{eq:nu}
    \mu=\mu_\phi, \quad \nu := \min \left\{ \mu_\phi \, , \, \inf_k \mu (\hat \Gamma_k) \right\}.
\end{equation}
${\rm (}$By convention, \eqref{complprimal} 
should be understood as being equal to $d_0^2/({\underline{\lambda}}{\bar{\varepsilon}})$ when $\mu=0$.${\rm )}$
\end{theorem}

\begin{proof}
It is straightforward to see from Lemma \ref{lem:sigma}(b) that $\sigma(\lam) \ge \underline \sigma := \sigma(\underline \lam)$ for $\lam \ge \underline \lam$, where $\sigma(\lam)$ denotes the $\sigma$ defined in \eqref{eq:sigma}, considered as a function of $\lam$.
Note that we invoke the oracle $(\hat x_k, \hat y_k, \hat \Gamma_k, \lambda_k) = \BB(\hat x_{k-1},\chi, (1-\chi) \bar \varepsilon/2)$ in step 1.
Using \eqref{ineq:contraction}, assumption (F1), and the observation that $\sigma(\lam_k) \ge \underline \sigma$, we have
\begin{equation}\label{ohiustb}
    2 \lambda_k\left[\phi(\hat y_k)-\phi(u)\right]  \le \lam_k \bar \varepsilon + \left\|\hat x_{k-1}-u\right\|^2 - (1+\underline \sigma) \|\hat x_k-u\|^2,
\end{equation}
%
%
where 
\begin{equation}\label{def:sigma}
    \underline \sigma = \frac{\underline{\lam}[\nu(1+\mu \underline{\lam})+\chi(\mu-\nu)]}{1+\mu {\underline \lambda}+\chi {\underline \lambda}(\nu-\mu)},
\end{equation}
and $\mu$ and $\nu$ are as in \eqref{eq:nu}.
It is easy to see that \eqref{ohiustb} with $u=x_*$, where $x_*$ denotes the closest solution of \eqref{eq:ProbIntro} to the initial point $\hat x_0$, satisfies 
$$
	\gamma_k \eta_k \le \alpha_{k-1} - (1+\sigma)\alpha_k + \gamma_k \delta
$$
with
\begin{equation}\label{eq:parameter}
    \gamma_k = 2\lam_k, \quad \eta_k=\phi(\hat y_k) - \phi_*, \quad \alpha_k = \|\hat x_k-x_*\|^2, \quad \delta = \frac{\bar \varepsilon}{2}, \quad \sigma = \underline \sigma, \quad {\underline \gamma} = 2{\underline \lam},
\end{equation}
where ${\underline \gamma}$
satisfies $\gamma_k \geq {\underline \gamma}$.

Using Lemma \ref{lm:easyrecur1}(b) with the above parameters, Lemma \ref{lem:sigma}(c), and the definition of $\underline \sigma$ in \eqref{def:sigma} that the complexity to find a $\bar \varepsilon$-solution is
\begin{equation}\label{eq:complexity}
  \min\left\{\frac{1+\underline \sigma}{\underline \sigma} \log\left(1+ \frac{\underline\sigma d_0^2}{{\underline \lambda}{\bar \varepsilon}}\right), \frac{d_0^2}{{\underline \lambda}{\bar \varepsilon}} \right\} \le \min\left\{\min\left[ \frac{1}{\chi}\left(1+\frac{1}{\underline{\lam} \mu}\right),  1+\frac{1}{\underline{\lam} \nu}\right] \log\left(1+ \frac{\underline \sigma d_0^2}{{\underline \lambda}{\bar \varepsilon}}\right), \frac{d_0^2}{{\underline \lambda}{\bar \varepsilon}} \right\}.  
\end{equation}
Moreover, we note that $\underline \sigma \le \underline \lam \mu$ in view of the second inequality in Lemma \ref{lem:sigma}(a).
Finally, this observation and \eqref{eq:complexity} immediately imply the complexity in \eqref{complprimal}.
\end{proof}

We now comment on the complexity bound obtained in Theorem \ref{thouter}.
First, under the assumption that $\mu>0$, the bound 
\begin{equation}\label{bound:min}
    \min\left\{ \frac{1}{\chi}\left(1+\frac{1}{\underline{\lam} \mu}\right),  1+\frac{1}{\underline{\lam} \nu}\right\} \log\left(1+ \frac{\mu d_0^2}{\bar \varepsilon}\right)
\end{equation}
implied by \eqref{complprimal} 
is meaningful only when either 
$\chi>0$ or $\nu>0$
(otherwise, it should be understood as being infinity). 
Second, the validity of the
second bound in \eqref{bound:min} has been established for some composite subgradient and proximal bundle methods (see for example \cite{du2017rate,liang2020proximal}).
Third, if $\mu\gg\nu$ and $\chi$ is sufficiently
 bounded away from zero, the smallest term in \eqref{bound:min} is the first one, in which case
\eqref{complprimal} reduces to
\[
\frac{1}{\chi}\left(1+\frac{1}{\underline{\lam} \mu}\right) \log\left(1+ \frac{\mu d_0^2}{\bar \varepsilon}\right)
= \tilde {\cal O}\left ( \frac{1}{\underline{\lam} \mu}\right).
\]
Fourth, since the second bound $d_0^2/(\underline \lambda \bar \varepsilon)$ does not depend on  $\nu$, $\mu$ and $\chi$,
it holds for any parameters
$\mu \ge \nu \ge 0$
and $\chi \in [0,1)$.

\subsection{Proof of Proposition \ref{propouter}}\label{subsec:contraction}

We need the following technical result before proving Proposition \ref{propouter}.

\begin{lemma}\label{lemetab} 
Define 
 \begin{align}
	 s &:= \frac{x^- - x}{\lam}, \label{defsb}\\
	 \eta &:= \phi(y) - \Gamma(x) - \inner{s}{y- x}. \label{defetab}
	\end{align}

Then, the following statements hold:
\begin{itemize}
\item[a)]
  $s \in \partial \Gamma (x)$, i.e., for every $u \in \R^n$,
 \begin{equation}\label{eq:subdifb}
    \Gamma(u) \ge \Gamma(x) + \inner{s}{u-x}; 
 \end{equation}  
  \item[b)] $s \in \partial_{\eta} \phi(y)$ and
  \begin{equation}\label{ineqetab}
  0 \le 2 \lam \eta \le  2 \lam \varepsilon - 
	    \|y- x\|^2 + (1-\chi)\|y- x^-\|^2.
     \end{equation}
\end{itemize}
\end{lemma}

 \begin{proof} 
 (a) The optimality condition of \eqref{defxhatk} yields $0 \in \partial \Gamma_k(x) +(x-x^-)/\lambda$, which together with the definition of $s$ in \eqref{defsb} implies that the inclusion in a) holds.
 Relation \eqref{eq:subdifb} immediately follows from the inclusion in a) and the definition of the subdifferential.
 
(b) The relation $\phi \geq \Gamma$ (see step 1 of FSCO), \eqref{eq:subdifb}, the definition of $\eta$ in \eqref{defetab} imply that for every $u\in \dom \phi$,
	\[
	\phi(u) \ge \Gamma(u) \stackrel{\eqref{eq:subdifb}}\ge  \Gamma(x) + \inner{s}{u- x}
      \stackrel{\eqref{defetab}}=  \phi(y) + \inner{s}{u- y} - \eta,
	\]
which yields the inclusion in b). 
Taking $u=y$ in the above inequality gives $\eta \ge 0$, and hence the first inequality in \eqref{ineqetab} holds.
Using the definitions of $s$ and $\eta$ in \eqref{defsb} and \eqref{defetab}, respectively, and \eqref{propyb} and \eqref{defxhatk}, we have
\[
\begin{array}{lcl}
    \eta &\stackrel{\eqref{defetab}}=& \displaystyle \phi(y) - \Gamma(x) - \inner{s}{y- x} \\
     & \stackrel{\eqref{propyb},\eqref{defxhatk}}{\le} &\displaystyle \left[ \varepsilon - \frac\chi{2\lam}\|y- x^-\|^2 + \frac{1}{2\lam}\|x- x^- \|^2 \right] - \inner{s}{y- x}\\
	    &\stackrel{\eqref{defsb}}{=} &  \displaystyle \varepsilon - \frac\chi{2\lam}\|y- x^-\|^2  +  \frac1{2\lam}\|x- x^-\|^2 + \Inner{\frac{x- x^-}{\lam}}{y- x} \\
     &=&\displaystyle  \varepsilon + \frac{1-\chi}{2\lam}
	    \|y- x^- \|^2 - \frac1{2\lam}
	    \|y- x\|^2.
	\end{array}
\]
Hence, the second inequality in \eqref{ineqetab} holds.
\end{proof}

Using a technical result, Lemma \ref{lem:app} from Appendix \ref{subsec:tech-FSCO}, we are able to partially recover the intrinsic convexity $\mu_\phi$, even though it is not assumed to be known from either the objective $\phi$ or the model $\Gamma$.

\begin{lemma}\label{lem:zeta}
    Let $\mu=\mu_\phi$ and assume that $\Gamma \in \nConv{n}$ for some $\nu \in [0,\mu]$.
    For any $\zeta \in [0,1)$ and every $u\in \dom \phi$, we have
     \begin{equation}\label{qphi-1}
    \phi(u) \ge \phi(y) + \inner{s}{u- y} + \frac{\zeta (\mu-\nu)}{2} \|u- y\|^2 - \frac{\eta}{1-\zeta} + \frac{\zeta}{1-\zeta} \frac{\nu}{2}\|y - x\|^2 + \frac{\nu}{2}\|u - x\|^2,
	\end{equation}
    where $s$ and $\eta$ are as in \eqref{defsb} and \eqref{defetab}, respectively.
\end{lemma}

\begin{proof}
    Define
\begin{equation}\label{eq:tilde}
    \tilde s = s + \nu(x^- - x), \quad \tilde \eta = \eta - \frac{\nu}{2}\|y - x\|^2, \quad \tilde \phi = \phi - \frac{\nu}{2}\|\cdot - x^-\|^2, \quad \tilde \Gamma = \Gamma - \frac{\nu}{2}\|\cdot - x^-\|^2.
\end{equation}
It follows from Lemma \ref{lemetab}(a) and the above definitions of $\tilde s$ and $\tilde \Gamma$ that
$\tilde s \in \partial \tilde \Gamma(x)$.
In view of \eqref{eq:tilde} and the assumption that $\Gamma \in \nConv{n}$ for some $\nu \in [0,\mu]$, we observe that $\tilde \Gamma$ is convex and $\tilde \phi \ge \tilde \Gamma$.
Hence, following a similar argument as in Lemma~\ref{lemetab}(b), we conclude that
$\tilde s \in \partial_{\tilde \eta} \tilde \phi(y)$.
Using this inclusion, the fact that $\tilde \phi$ is 
$(\mu -\nu)$-convex, and Lemma \ref{lem:app} with $(\psi,\xi,y,v,\eta,Q)=(\tilde \phi, \mu-\nu,y,\tilde s, \tilde \eta,I)$, we have for any $\tau\in (0,\infty]$ and every $u\in \dom \phi$,
\[
\tilde \phi(u) \stackrel{\eqref{ineq:app}}\ge \tilde \phi(y) + \inner{\tilde s}{u- y} + \frac{\mu -\nu}{2(1+\tau)} \|u- y\|^2 -  (1+\tau^{-1}) \tilde \eta.
\]
Taking $\tau=(1-\zeta)/\zeta$ for $\zeta \in [0,1)$ in the above inequality, we obtain
\[
     \tilde \phi(u) \ge \tilde \phi(y) + \inner{\tilde s}{u- y} + \frac{\zeta (\mu -\nu)}{2} \|u- y\|^2 -   \frac{\tilde \eta}{1-\zeta}.
	\]
Finally, \eqref{qphi-1} follows from using the definitions of $\tilde \phi$, $\tilde s$, and $\tilde \eta$ in \eqref{eq:tilde}, and rearranging the terms in the above inequality.
\end{proof}

Noting that in the proof of Lemma \ref{lem:zeta}, we apply Lemma~\ref{lem:app} to $\tilde \phi$ in \eqref{eq:tilde}. If we instead apply Lemma~\ref{lem:app} to $\phi$ directly, we obtain for every $u \in \dom \phi$,
\begin{equation}\label{ineq:direct}
    \phi(u) \ge \phi(y) + \inner{s}{u- y} + \frac{\zeta \mu}{2} \|u- y\|^2 - \frac{\eta}{1-\zeta}.
\end{equation}
This inequality is weaker than the one in \eqref{qphi-1}. Indeed, applying the Cauchy-Schwarz inequality, we have
\[
\left(\frac{1-\zeta}{\zeta} + 1\right)\left(\frac{\zeta}{1-\zeta} \frac{\nu}{2}\|y - x\|^2 + \frac{\nu}{2}\|u - x\|^2\right) \ge \frac\nu 2(\|y - x\|+ \|u - x\|)^2 \ge \frac\nu 2\|u - y\|^2,
\]
where the last inequality is due to the triangle inequality. Hence, we have
\[
\frac{\zeta}{1-\zeta} \frac{\nu}{2}\|y - x\|^2 + \frac{\nu}{2}\|u - x\|^2 \ge \frac{\zeta\nu}2 \|u - y\|^2,
\]
and thus \eqref{qphi-1} implies \eqref{ineq:direct}.


If $\zeta$ is chosen to be $0$, then \eqref{ineq:direct} becomes
\[
\phi(u) \ge \phi(y) + \inner{s}{u- y} - \eta \quad \forall u \in \dom \phi,
\]
which is equivalent to the inclusion in Lemma \ref{lemetab}(b). This means that if we use \eqref{ineq:direct}, then we cannot recover the intrinsic convex parameter $\mu_\phi$ when $\zeta=0$. However, even if $\zeta=0$, \eqref{qphi-1} still preserves a $\nu$-strongly convex lower approximation to $\phi$. This justifies the necessity of the assumption that $\Gamma \in \nConv{n}$.

Now we are ready to present the proof of Proposition \ref{propouter}.

\vgap
\noindent
{\bf Proof of Proposition \ref{propouter}.}
It follows from the definition of $s$ in \eqref{defsb} and Lemma \ref{lem:zeta} with $\zeta =\chi$,
\begin{align}
&\|x^--u \|^2  - \|x-u\|^2 -\left(\|x^--y\|^2 - \|x-y\|^2\right) 
= -2\lam \inner{s}{u-y} \nn \\
\stackrel{\eqref{qphi-1}}\ge & 2 \lambda \left[\phi(y)- \phi(u)\right] + \chi (\mu -\nu) \lam \|u-y\|^2 -   \frac{2\lam \eta}{1-\chi} + \frac{\chi \nu \lam}{1-\chi} \|y - x\|^2 + \nu \lam\|u - x\|^2. \label{ineq:long}
\end{align}
Rearranging the terms in \eqref{ineq:long} and using Lemma \ref{lemetab}(b), we have
\begin{align}
&\left\|x^--u \right\|^2  - (1+\nu \lam)\|x-u\|^2  -2 \lambda \left[\phi(y)- \phi(u)\right] \nn \\
& \stackrel{\eqref{ineq:long}}\ge \|x^--y\|^2 + \left( \frac{\chi \nu \lam}{1-\chi} - 1 \right)\|x-y\|^2 + \chi (\mu -\nu) \lam \|u-y\|^2 -   \frac{2\lam \eta}{1-\chi} \nn \\
& \stackrel{\eqref{ineqetab}}\ge  -\frac{2 \lam \varepsilon}{1-\chi} + \frac{\chi(1+\nu \lambda)}{1-\chi} \|x-y\|^2 + \chi (\mu -\nu) \lam \|u-y\|^2. \label{ineq:imp}
\end{align}
Rearranging the terms in the above inequality, we have
\begin{align}
    2\lambda \left[\phi(y)- \phi(u)\right] \le& 
\frac{2 \lam \varepsilon}{1-\chi} + \left\|x^--u \right\|^2  - (1+\nu \lam)\|x-u\|^2 \nn \\
&- \chi \left(\frac{1+\nu \lam}{1-\chi} \|x-y\|^2 + (\mu -\nu) \lam \|u-y\|^2\right). \label{telescopic1}
\end{align}
Using the triangle inequality and the fact that $(a_1+a_2)^2 \le (b_1^{-1}+b_2^{-1}) (a_1^2 b_1 + a_2^2 b_2)$ with $(a_1,a_2) =(\|x-y\|,\|u-y\|)$ and $(b_1,b_2) =((1+\nu \lam)/(1-\chi), (\mu -\nu) \lam)$
we have
\begin{equation}\label{ineqxku}
    \|x-u\|^2 \le \left(\frac{1-\chi}{1+\nu \lam} + \frac{1}{(\mu -\nu)\lam}\right) \left(\frac{1+\nu \lam}{1-\chi} \|x-y\|^2 + (\mu -\nu) \lam \|u-y\|^2\right).
\end{equation}
Plugging the above inequality into \eqref{telescopic1}, we have
\[
2\lambda \left[\phi(y)-\phi(u)\right] 
\le   \frac{2 \lam \varepsilon}{1-\chi} +  
\left\|x^--u\right\|^2  - 
\left[ 1 + \nu \lam + \chi \left(\frac{1-\chi}{1+\nu \lam} + \frac{1}{(\mu -\nu)\lam}\right)^{-1} \right]
\|x-u\|^2,
\]
which is the same as \eqref{ineq:contraction} after simplification.
\QEDA

\section{Proofs of Theorems \ref{thoutercs} and \ref{complprimal-2U-PBt}}\label{sec:pf1}

\subsection{Proof of Theorem \ref{thoutercs}}\label{subsec:proof-CS}

The following result shows that U-CS 
is an instance of FSCO and that
assumption (F1)
of Section~\ref{sec:FSCO} is satisfied.

\begin{proposition}\label{lem:ACSCS}
The following statements hold for U-CS:
    \begin{itemize}
    \item[a)] $\{\lam_k\}$ is a non-increasing sequence;
        \item[b)] for every $k\ge 1$, we have
        \begin{align}
		&\hat x_k =\underset{u\in  \R^n}\argmin
		\left\lbrace  \ell_f(u;\hat x_{k-1})+h(u) +\frac{1}{2\lam_{k}} \|u- \hat x_{k-1} \|^2 \right\rbrace, \label{eq:sub1cs} \\
		&f(\hat x_{k}) -\ell_f(\hat x_{k};\hat x_{k-1})+\frac{\chi-1}{2 \lambda_{k}}\|\hat x_{k}-\hat x_{k-1}\|^2 \le \frac{(1-\chi)\bar \varepsilon}2,
  \label{ineq:iterationcs}\\
   & \lam_k \ge \min \left \lbrace \frac{(1-\chi)^2 \bar \varepsilon}{8 M_f^2 + 2 \bar \varepsilon L_f}, \lam_0 \right\rbrace.\label{ineq:lamj2}
	    \end{align}

       \item[c)]     U-CS is a special case of FSCO where:
    \begin{itemize}
        \item [i)]
        $\hat y_k=\hat x_k$ and $\hat \Gamma_k(\cdot)=\ell_f(\cdot;\hat x_{k-1})+h(\cdot)$
    for every $k \ge 1$;
    \item[ii)] assumption (F1) is satisfied with
    $\underline \lam$ given by \eqref{ineq:lamj2}.
    \end{itemize}

    \end{itemize}
    \end{proposition}    
    \begin{proof}
    a) This statement directly follows from the description of U-CS.
    
    b) Relations \eqref{eq:sub1cs} and \eqref{ineq:iterationcs} directly follow from the description of U-CS.
Next, we prove \eqref{ineq:lamj2}.
Supposing $\lam_{k-1}\le (1-\chi)^2 \bar \varepsilon/(4 M_f^2 + \bar \varepsilon L_f)$,
then it follows from Lemma \ref{lem:lam} that \eqref{ineq:iterationcs} holds with $\lam_k$ replaced by $\lambda_{k-1}$. This indicates that if $\lam$ is small enough, then it will remain unchanged. Therefore, following from the update scheme of $\lam$ in step 2 of U-CS, there is a lower bound as in \eqref{ineq:lamj2}.

c) Relations
  \eqref{eq:sub1cs}  and
     \eqref{ineq:iterationcs}
     are the analogues of
     relations \eqref{propyb} and  \eqref{defxhatk}
     of FSCO with 
$\hat y_k=\hat x_k$ and $\hat \Gamma_k$  as in (i).
Inequality \eqref{ineq:lamj2}
shows that Assumption (F1) is
satisfied. 
    \end{proof}

\vgap

\noindent
{\bf Proof of Theorem \ref{thoutercs}.}
Define 
\begin{equation}\label{defkbar}
\bar k=\left \lceil 2 \log \max\left\{\frac{ \lambda_0 (8M_f^2 + 2 \bar \varepsilon L_f) }{(1-\chi)^2 \bar \varepsilon},1\right\} \right \rceil.
\end{equation}
Observe that $\lambda_0/2^{k} \leq {\underline \lambda}$ for $k \geq \bar k$ where ${\underline{\lambda}}$ is given by the right-hand side of \eqref{ineq:lamj2},
and from Proposition \ref{lem:ACSCS} that we cannot halve $\lambda$ more than $\bar k$ iterations. 
It follows from \eqref{defkbar} and the definition of $ Q(\bar \varepsilon)$ in \eqref{defQf} that
\[
\bar k \leq  \displaystyle \left \lceil 
2 \log \frac{\lambda_0 Q({\bar \varepsilon})}{\bar \varepsilon} \right \rceil.
\]
Therefore, the second term on the right-hand side of \eqref{complprimal-2}
gives an upper bound on the number of iterations with backtracking of $\lambda$.
We now provide a bound on the number of remaining iterations to obtain a $\bar \varepsilon$-solution with U-CS.
In the $k$-th iteration of U-CS, the model $\hat \Gamma_k$ for $\phi=f+h$ is simply $\ell_f(\cdot,\hat x_{k-1})+h(\cdot)$ where $\ell_f$ is the linearization of $f$ as in \eqref{linfdef}, hence $\mu(\hat \Gamma_k) = \mu_h \le \mu_\phi$ and $\nu=\mu_h$ in \eqref{eq:nu}.
Since Proposition \ref{lem:ACSCS} shows that U-CS is a special case of FSCO, Theorem~\ref{thouter} immediately gives for U-CS the upper bound \eqref{complprimal} with ${\underline{\lambda}}$ given by \eqref{ineq:lamj2}
on the number of the remaining iterations
(where $\lambda$ is not halved)
required to find a $\bar \varepsilon$-solution 
of \eqref{eq:ProbIntro2}.
Using the inequality
\begin{equation}\label{upperlbar}
\frac{1}{\underline{\lam}} \stackrel{\eqref{ineq:lamj2}}=\max \left \lbrace \frac{8 M_f^2 + 2 \bar \varepsilon L_f}{(1-\chi)^2 \bar \varepsilon}, \frac{1}{\lambda_0} \right\rbrace
\leq \frac{8 M_f^2 + 2 \bar \varepsilon L_f}{(1-\chi)^2 \bar \varepsilon} + \frac{1}{\lambda_0}
\end{equation}
and the definition of $ Q(\bar \varepsilon)$ in \eqref{defQf},
we conclude that
$1/\underline{\lam} \le 
Q(\bar \varepsilon)/\bar \varepsilon$.
This observation and \eqref{complprimal} with ${\underline{\lambda}}$ given by \eqref{ineq:lamj2} thus imply that the first term in \eqref{complprimal-2} is an upper bound on the number of the remaining iterations
(where $\lambda$ is not halved).
This completes the proof.
\QEDA

\subsection{Proof of Theorem \ref{complprimal-2U-PBt}} \label{sec:proofupb}

We begin with a technical result, i.e., Lemma \ref{lem:tj}, 
uses arguments similar to ones
used in Subsection 4.2 of \cite{liang2023unified}. For the sake of completeness, its proof is given in Appendix~\ref{appendixupb}.
These two results together will be used in the proof of Lemma \ref{lemserious} 
to show that if stepsize $\lambda$ in a cycle
is sufficiently small, then the cycle must be a serious one.

\begin{lemma}\label{lem:tj}
If $j$ is an iteration of a  cycle 
such that $t_j>(1-\chi){\bar \varepsilon}/2$, then
\begin{equation}\label{ineq:tj-recur2}
     j-i < (1+u_\lambda)
     \log\left( \frac{4  t_i}{(1-\chi)\bar \varepsilon} \right),
     \end{equation}
     where $\lam$ is the prox stepsize for the cycle, $i$ is the first iteration of the cycle, $t_i$ is as in \eqref{ineq:hpe1}, and
     \begin{equation}\label{def:ulam}
         u_{\lambda}:=\frac{4\lam[2 M_f^2 + (1-\chi) \bar \varepsilon L_f]}{(1-\chi) \bar \varepsilon}.
     \end{equation}
     \end{lemma}    

Noting that the right-hand side of \eqref{ineq:tj-recur2} depends on $t_i$, the next result provides two upper bounds
on this quantity,
one that does not require (A4) but only  holds for  $\lam$ sufficiently small,
and one that
requires (A4) and holds for an arbitrary $\lam>0$.
The proof of the first one is well-known for $\chi=0$ (e.g., Lemma 5.5 in \cite{liang2024single})
but not for the second one. The proofs for both of them are given in Appendix \ref{appendixupb}.

\begin{lemma}\label{lem:tj2}
 Let $i$ denote the first iteration of a cycle of U-PB. The following statements about the quantity $t_i$ hold:
 \begin{itemize}
\item[a)] if  $\lambda \leq (1-\chi)/(2{  L_f})$, then
\begin{equation}\label{ineq:ti}
    t_i \le \frac{4\lam  M_f^2}{1-\chi};
\end{equation}
\item[b)] if  (A4) holds, then
\begin{equation}\label{ineq:ti2}
    t_i \le \frac{\lam(16 M_f^2 + L_f^2 D^2)}{4(1-\chi)}.
\end{equation}
 \end{itemize}
	\end{lemma}


Recall that U-PB once in a while performs reset cycles, i.e., cycles which perform exactly $\overline N$ inner iterations and terminate without any of their iterates satisfying the condition $t_j \le (1-\chi)\bar \varepsilon/2$.
If the parameter pair $(M_f,L_f)$ is known, then the result below shows that, if $\lam$ is chosen sufficiently small, then no reset cycle is performed.
On the other hand, if  $(M_f,L_f)$ is not known, then  U-PB performs a line search on $\lam$ and may generate reset cycles.
Reset cycles are thrown away by U-PB, but they signal that the current $\lam$ is too large and hence should be decreased (see step 3.b of U-PB).

As already observed in the first remark in the second paragraph following U-PB,
U-CS is a special case of U-PB in which  $\overline N=1$.
In this regard,
the following technical result is a generalization of Lemma \ref{lem:lam}
to the context
of U-PB.
\begin{lemma}\label{lemserious} 
Assume that assumption (A4) holds. If for some cycle of U-PB, 
 the prox stepsize $\lambda$  satisfies
\begin{equation}\label{condlambda}
 \lam \le 
 \frac{(1-\chi)^2 {\overline  N} \bar \varepsilon}{2({  M}_f^2 + \bar \varepsilon {  L}_f) B } 
\end{equation}
where $B $ is as in \eqref{defa}, then the cycle must be a serious one.
As a consequence, if the initial prox stepsize $\lam_0$ is chosen less than or equal to the right-hand side of \eqref{condlambda}, then U-PB has the following properties: a) it keeps $\lam$ constant throughout its execution; and b) it never performs a reset update.
\end{lemma}

\begin{proof} 
Assume for contradiction that the cycle is not a serious one and its stepsize $\lam$ satisfies \eqref{condlambda}.
Then, this cycle is a reset cycle and hence performs exactly
$\overline N$ iterations.
Let $i$ denote the first iteration of the cycle.
In this proof, to alleviate notation, we denote
$\varepsilon=(1-\chi)\bar \varepsilon/2$.
Then by 
inequality \eqref{ineq:tj-recur2} in 
Lemma
\ref{lem:tj} written for $j=i+{\overline N}-1$, we must have $\overline N-1 < (1+u_{\lambda})
     \log (2  t_i/\varepsilon)$
 where $u_\lambda$ is as in \eqref{def:ulam}.
Hence, to obtain a contradiction and prove the lemma, it suffices to show that 
\begin{equation}\label{defutildeine2}
(1+u_{\lambda}) \log \frac{2 t_i}{\varepsilon} \leq {\overline N}-1.
\end{equation}
To show the above inequality, 
we consider the following two cases: a)
$8\overline N \le  B$
and b)
$8\overline N>B $.

a) Assume that $8\overline N \le  B $. The assumption \eqref{condlambda} and the fact that $\varepsilon=(1-\chi)\bar \varepsilon/2$ then
imply that
\[
\lambda \leq 
 \frac{(1-\chi) \varepsilon}{8(M_f^2 +\varepsilon L_f)}
\leq \frac{1-\chi}{8 L_f}
\]
and hence that the assumption of Lemma 
\ref{lem:tj2}(a) holds.
Thus, statement (a) of this same lemma implies that 
\begin{equation}\label{lemti}
\frac{2 t_i}{\varepsilon} \stackrel{\eqref{ineq:ti}}{\le} \frac{8\lam   M_f^2 }{(1-\chi)\varepsilon} 
\stackrel{\eqref{condlambda}}{\leq}  
 \frac{8M_f^2 \overline N}{(M_f^2 + \varepsilon L_f) B }
 \le
 \frac{M_f^2} {M_f^2 + \varepsilon L_f} \leq 1
\end{equation}
where the second last inequality follows from the case assumption 
$8{\overline N} \le  B $.
Since the above inequality implies that
$\log(2 t_i/\varepsilon) \leq 0$, inequality
 \eqref{defutildeine2}
 immediately holds for case a).

b) Assume now that
$  8{\overline N} >  B $.
To show  \eqref{defutildeine2}, 
we first bound $1+u_\lambda$ from above as follows:
\begin{align}
1+u_\lambda &  
\stackrel{\eqref{def:ulam}}{=} 1+\frac{4\lam({  M}_f^2 + \varepsilon {  L}_f)  }{\varepsilon}  \stackrel{\eqref{condlambda}}{\leq} 
 1+\frac{4(1-\chi)\overline N } {B } \nn\\
&\leq \frac{8\overline N}{B } + \frac{4\overline N}{B } = \frac{12\overline N}{B } \label{bound1pu}
\end{align}
where we also use the fact that $\varepsilon=(1-\chi)\bar \varepsilon/2$ in the first line and the second last inequality is due to the case assumption $8{\overline N}>B$.
Now observe that condition \eqref{condlambda}
on $\lambda$
implies
\begin{align}\label{condtilupp}
\frac{2t_i}{\varepsilon}
\stackrel{\eqref{ineq:ti2}}{\leq} 
 \frac{8   M_f^2 + {  L}_f^2 D^2/2}{(1-\chi) \varepsilon} \lambda &
  \stackrel{\eqref{condlambda}}{\le} \frac{8   M_f^2 + {  L}_f^2 D^2/2}{(1-\chi)\varepsilon}  \left[\frac{(1-\chi){\overline N} \varepsilon}{({  M}_f^2 + \varepsilon {  L}_f) B}\right]  
  \le \frac{
  (  M_f^2 + {L}_f^2 D^2/16)\overline N}{ M_f^2+\varepsilon L_f }
  \leq  1 + \frac{  {L}_f^2 D^2 \overline N}{16{ M}_f^2 } \nn 
  \end{align}
  where the second last inequality follows from the fact that $B  \geq 8$.
The above inequality and the definition of $B $ in \eqref{defa} then imply that
  \begin{equation}\label{finalcasebok}
  \log \frac{2t_i}{\varepsilon} \le
  \log \left(1 + \frac{  {L}_f^2 D^2 \overline N}{16{ M}_f^2 } \right) \stackrel{\eqref{defa}}= \frac{B -8}{12} \
  \leq \frac{B -B /\overline N}{12} =   \frac{B (\overline N-1)}{12\overline N}
 \end{equation}
where the last inequality follows from the case assumption $8{\overline N}>B$.
Finally, combining 
\eqref{bound1pu} and \eqref{finalcasebok}, we conclude that 
\eqref{defutildeine2} also holds in case b).
\end{proof}


The next result, which is the analogue of Proposition \ref{lem:ACSCS}, shows that the sequence of stepsizes $\lambda_k$ is bounded below by some positive constant and provides an upper bound on the number of reset cycles.

\begin{proposition}\label{lem1:U-PB}
The following statements hold for U-PB:
    \begin{itemize}
        \item[a)] every cycle in U-PB has at most $\overline N$ inner iterations;
        \item[b)] each stepsize $\lambda_k$ generated by U-PB satisfies 
  \begin{equation}\label{ineq:lamjag}
\lambda_k \geq 
\frac{\overline N \bar\varepsilon}{U(\bar \varepsilon)}
        \end{equation}  
        where $U(\bar \varepsilon)$ is as in \eqref{defU};
\item[c)] the number of reset cycles is upper bounded by
\begin{equation}\label{upperlambda}
\displaystyle \left \lceil 
2 \log \frac{\lambda_0 U(\bar \varepsilon)}{\overline N \bar \varepsilon} \right \rceil.
\end{equation}
    \end{itemize}
\end{proposition}

\begin{proof}
a) This statement immediately follows from the description of U-PB (see its step 3).

b) It follows from Lemma \ref{lemserious} and the update rule for $\lambda$ in U-PB (see its step 3) that
\[
        \lam_k \ge \min\left\{\frac{(1-\chi)^2 {\overline  N} \bar \varepsilon}{4(M_f^2 + \bar \varepsilon L_f) B }, \lam_0\right\},
        \]
and hence that
        \[
        \frac{1}{\lam_k} \le \max\left\{\frac{4(M_f^2 + \bar \varepsilon L_f) B }{(1-\chi)^2 {\overline  N} \bar \varepsilon}, \frac1{\lam_0}\right\} \le \frac{4(M_f^2 + \bar \varepsilon L_f) B }{(1-\chi)^2 {\overline  N} \bar \varepsilon} + \frac1{\lam_0}.
        \]
Thus, using the definition of $U(\bar \varepsilon)$ in \eqref{defU}, we conclude that
        \[
        \frac{\overline N \bar \varepsilon}{\lam_k} \le \frac{4(M_f^2 + \bar \varepsilon L_f) B }{(1-\chi)^2} + \frac{ \overline N \bar \varepsilon}{\lam_0} = U(\bar \varepsilon),
        \]
and hence that statement b) holds.

c) This statement immediately follows from b) and the update rule of $\lam$ in U-PB.
 \end{proof}

\if{
Suppose that there exists $\lam_k < \underline{\lam}$
and assume wlog that
$k$ is the smallest one. Then
$\lam_{k-1} \ge \underline{\lam}$.
The contradiction comes by showing that
$\lam_k = \lam_{k-1}$.
}\fi

The following result shows
that serious iterations of U-PB
generate sequences $\{\hat x_k\}$,
$\{\hat y_k\}$, $\{\lambda_k\}$,
and $\{\hat \Gamma_k\}$ satisfying the
requirements of FSCO.

        \begin{proposition}\label{lem1:U-PB1}
    U-PB is a special case of FSCO in the sense that:
   \begin{itemize} 
   \item[a)] the quadruple $(\hat x_k,\hat y_k,\hat \Gamma_k,\lam_k)$  satisfies relations 
   \eqref{propyb} and
\eqref{defxhatk} with $x^{-}=\hat x_{k-1}$;
\item[b)] condition (F1) used in the analysis of FSCO is satisfied.
\end{itemize}
\end{proposition}

\begin{proof} a) For a serious cycle, the pair $(\hat x_k,\hat y_k)$ of U-PB satisfies
\begin{align}
	&\hat x_k = \underset{u\in\R^n}\argmin \left\lbrace \hat \Gamma_k(u) + \frac{1}{2\lam_k}\|u- \hat x_{k-1} \|^2 \right\rbrace, \label{defxhatk2}\\
        &\phi(\hat y_{k})+\frac{\chi}{2 \lambda_{k}}\|\hat y_{k}-\hat x_{k-1}\|^2 - \left[\hat \Gamma_k(\hat x_k) + \frac{1}{2\lam_k}\|\hat x_k- \hat x_{k-1} \|^2\right] \le \frac{(1-\chi)\bar \varepsilon}2.
        \label{propyb2}  
\end{align}
Indeed, relations \eqref{defxhatk2} and \eqref{propyb2}
(which are \eqref{defxhatk} and \eqref{propyb} in FSCO, respectively) follow from \eqref{def:xj} and $t_{j} \leq (1-\chi)\bar \varepsilon/2$ with 
    $\hat x_k = x_{j}$, $\hat \Gamma_k=f_j+h$, $\hat y_k=y_j$, and $\lambda_k=\lambda$ (see the serious update in step 3.c of U-PB).

b) Condition (F1) is satisfied with
$\underline \lam = \overline N \bar\varepsilon / U(\bar \varepsilon)$,
in view of Proposition \ref{lem1:U-PB}(b).
\end{proof}

\vgap

\noindent
\textbf{Proof of Theorem \ref{complprimal-2U-PBt}} To alleviate notation, denote 
$\varepsilon=(1-\chi)\bar \varepsilon/2$ and $(\mu,\nu)=(\mu_\phi,\mu_h)$.
In view of step 3.c of U-PB, we have $\mu(\hat  \Gamma_k) = \mu_h = \nu$ and hence it is the same as the $\nu$ in \eqref{eq:nu}.
Noting that, by Proposition~\ref{lem1:U-PB1}, the sequences indexed by $k$ is a special implementation of FSCO, and that,  by the proof of Proposition~\ref{lem1:U-PB}(b), (F1) holds with $\underline \lam = \overline N \bar\varepsilon / U(\bar \varepsilon)$, it follows from Theorem \ref{thouter} that the total number of serious cycles generated by U-PB
is bounded by
\[
\min\left\{\min\left[ \frac{1}{\chi}\left(1+  \frac{U({\bar \varepsilon})}{\mu \bar \varepsilon \overline N} \right),  1+\frac{U({\bar \varepsilon})}{\nu \bar \varepsilon \overline N}\right]
\log \left(1 + \frac{\mu d_0^2}{\bar \varepsilon}\right) \, , \, \frac{d_0^2 U({\bar \varepsilon})}{\bar \varepsilon^2 \overline N} \right\}.
\]
The conclusion of the theorem follows from the above bound and statements (a) and (c) of Proposition~\ref{lem1:U-PB}.
\QEDA


\section{Concluding remarks} \label{sec:conclusion}

In this paper, we presented two $\mu_\phi$-universal methods, namely U-CS and U-PB, to solve HCO \eqref{eq:ProbIntro2}. 
We propose FSCO to analyze both methods in a unified manner and establish functional  complexity bounds. 
We then prove that both U-CS and U-PB are instances of FSCO and apply the complexity bounds for FSCO to obtain iteration complexities for the two methods.
The two proposed methods are completely universal with
respect to all problem parameters, in particular, the strong convexity. Moreover, they do not require knowledge of the optimal value, and do not rely on any restart or parallelization schemes (such as the ones used in \cite{renegar2022simple}).


Some papers about universal methods (see for example \cite{lan2015bundle,nesterov2015universal}) assume that, for some $\alpha \in [0,1]$, $f$ in \eqref{eq:ProbIntro2} has $\alpha$-Hölder continuous gradient, i.e.,  there exists $L_\alpha \ge 0$ such that
    $\|\nabla f(x) - \nabla f(y)\| \le L_\alpha \|x-y\|^\alpha$
for every $x,y \in \dom h$.
It is shown in \cite{nesterov2015universal} that the universal primal gradient method proposed on it
(i.e., the U-CS method with $\chi=0$) finds a $\bar \varepsilon$-solution of \eqref{eq:ProbIntro2} in 
\begin{equation}\label{bound:UPG}
    \tilde {\cal O}\left( \frac{d_0^2 L_\alpha^{\frac{2}{\alpha+1}}}{\bar \varepsilon^{\frac{2}{\alpha+1}}}\right)
\end{equation}
iterations.
This result also follows as a consequence of our results in this paper. Indeed, first note that the dominant term in the iteration complexity \eqref{complprimal-2} for the U-CS method is $\tilde {\cal O}(d_0^2(M_f^2 + \bar \varepsilon L_f)/\bar \varepsilon^2)$.
Second, Proposition 2.1 of \cite{liang2023unified} implies
that there exists a pair $(M_f,L_f)$ satisfying (A2) and the inequality
\[
M_f^2 + \bar \varepsilon L_f \le 2 \bar \varepsilon^{\frac{2\alpha}{\alpha+1}} L_\alpha^{\frac{2}{\alpha + 1}}.
\]
Hence, it follows from these two observations  that  \eqref{complprimal-2} is  sharper than the bound \eqref{bound:UPG} obtained in \cite{nesterov2015universal}.



We finally discuss some possible extensions of our analysis in this paper.
First, it is shown in Theorem~\ref{thoutercs}  (resp., Theorem~\ref{complprimal-2U-PBt})
that U-CS (resp., U-PB) is $\mu_\phi$-universal if $\chi>0$ and is $\mu_h$-universal if $\chi=0$. It would be interesting to investigate whether they are also 
$\mu_\phi$-universal for $\chi=0$. Note that this question is related to whether
the universal primal gradient of \cite{nesterov2015universal} (which is the same as U-CS with $\chi=0$) is $\mu_\phi$-universal.
Second, it would also be interesting to study whether the general results obtained for the FSCO framework can also be used to show
that other methods for solving the HCO 
problem \eqref{eq:ProbIntro2} are $\mu_\phi$-universal.
Third, a more challenging question is whether a 
$\mu_\phi$-universal method for solving \eqref{eq:ProbIntro2} with a checkable termination condition (and hence which does not depend on $\phi_*$) can be developed.
Finally, this work has only dealt with unaccelerated methods for solving \eqref{eq:ProbIntro2}. It would be interesting to develop
accelerated methods such as those for when $\mu=0$ (e.g., in  \cite{lan2015bundle,nesterov2015universal}), which are $\mu_\phi$-universal in the hybrid setting of this paper.\\

\par {\textbf{Data Availability statement.}} There is no data related to this publication.



\bibliographystyle{plain}
\bibliography{ref}

\appendix

\section{Technical results}\label{sec:appendix}

This appendix contains two subsections. Subsection \ref{subsec:tech-FSCO}
presents some technical results used in the analysis of FSCO while Subsection \ref{appendixupb}
proves Lemmas \ref{lem:tj} and \ref{lem:tj2}.

\subsection{Miscellaneous technical results}\label{subsec:tech-FSCO}

The following technical result is used in the proof of Lemma \ref{lem:zeta}.

\begin{lemma}\label{lem:app}
Assume that $\xi>0, \psi \in \bConv{n}$ and $Q \in \mathcal{S}_{++}^n$ are such that $\psi-$ $(\xi / 2)\|\cdot\|_Q^2$ is convex and let $(y, v, \eta) \in \R^n \times \R^n \times \R_{+}$be such that $v \in \partial_\eta \psi(y)$. Then, for any $\tau>0$,
    \begin{equation}\label{ineq:app}
        \psi(u) \ge \psi(y)+\inner{v}{u-y}-\left(1+\tau^{-1}\right) \eta+\frac{(1+\tau)^{-1} \xi}{2}\|u-y\|_Q^2 \quad \forall u \in \R^n .
    \end{equation}
\end{lemma}
\begin{proof}
    The proof of this result can be found in Lemma A.1 in \cite{melo2023proximal}.
\end{proof}

The next technical result is used in the proof of Theorem \ref{thouter}.

\begin{lemma} \label{lm:easyrecur1}
		Assume that sequences $\{\gamma_j\}$, $\{\eta_j\}$, and  $\{\alpha_j\}$ satisfy for every $j\ge 1$, $\gamma_j \ge {\underline \gamma}$ and 
		\beq \label{eq:easyrecur1}
		\gamma_j \eta_j \le \alpha_{j-1} - (1+\sigma)\alpha_j + \gamma_j \delta
		\eeq
  for some $\sigma \ge 0$, $\delta \ge 0$ and ${\underline \gamma} > 0$. Then, the following statements hold:
	\begin{itemize}
            \item[a)] for every $k\ge 1$,
            \begin{equation}\label{ineq:eta}
            \min_{1\le j \le k}  \eta_j \le \frac{\alpha_0 - (1+\sigma)^k \alpha_k}{\sum_{j=1}^k (1+\sigma)^{j-1} \gamma_j} + \delta;
        \end{equation}
		\item[b)]  if the sequence $\{\alpha_j\}$ is nonnegative, then $ \min_{1\le j \le k} \eta_j \le 2\delta $ for every $ k\ge 1$ such that
		\[
		k \ge \min \left\lbrace  \frac{1+\sigma}{\sigma} \log\left( \frac{\sigma \alpha_0}{{\underline \gamma}\delta} + 1 \right), \frac{\alpha_0}{{\underline \gamma}\delta} \right\rbrace
		\]
		with the convention that the first term is equal to the second term
		when $\sigma=0$. (Note  that the first term converges to the second term
		as $\sigma \downarrow 0$.)
	\end{itemize}
\end{lemma}

\begin{proof}
	a) Multiplying \eqref{eq:easyrecur1} by $ (1+\sigma)^{j-1} $ and summing the resulting inequality from $ j=1 $ to $ k $, we have
	\begin{align}
	    \sum_{j=1}^k (1+\sigma)^{j-1} \gamma_j \left[ \min_{1\le j \le k}  \eta_j \right] &\le \sum_{j=1}^k (1+\sigma)^{j-1} \gamma_j \eta_j
     \le  \sum_{j=1}^k (1+\sigma)^{j-1} \left(  \alpha_{j-1} - (1+\sigma) \alpha_j + \gamma_j \delta \right) \nn \\
		 &=  \alpha_0 - (1+\sigma)^k \alpha_k +
		\sum_{j=1}^k (1+\sigma)^{j-1}\gamma_j \delta. \label{ineq:sum}
	\end{align}
        Inequality \eqref{ineq:eta} follows immediately from the above inequality.

        
        b) It follows from \eqref{ineq:eta}, and the facts that $\alpha_k\ge 0$ and $\gamma_j \ge {\underline \gamma}$ that
        \begin{equation}\label{ineq:eta1}
            \min_{1\le j \le k}  \eta_j  \le \frac{\alpha_0}{{\underline \gamma} \sum_{j=1}^k (1+\sigma)^{j-1}} + \delta.
        \end{equation}
	Using the fact that $ 1+\sigma \ge e^{\sigma/(1+\sigma)} $
	for every $\sigma \ge 0$, 
	we have
	\begin{equation}\label{ineq:theta}
	    \sum_{j=1}^{k} (1+\sigma)^{j-1} = \max \left\lbrace \frac{(1+\sigma)^k-1}{\sigma}, k \right\rbrace \ge \max\left\lbrace \frac{e^{\sigma k /(1+\sigma)} - 1}{\sigma}, k\right\rbrace.
	\end{equation}
	Plugging the above inequality into \eqref{ineq:eta1}, we have for every $ k\ge 1 $,
	\[
	\min_{1\le j \le k}  \eta_j \le \frac{\alpha_0}{{\underline \gamma}}\min \left\lbrace \frac{\sigma}{e^{\sigma k /(1+\sigma)} - 1}, \frac 1k \right\rbrace + \delta,
	\]
	which can be easily seen to imply (b).
\end{proof}

The following technical result is used in the analysis of Section \ref{sec:FSCO}.

\begin{lemma}\label{lem:sigma}
  For every $\chi \in [0,1)$, $\lam>0$, $\mu>0$, and $\nu \in [0,\mu]$, the 
  quantity $\sigma=\sigma(\mu,\nu;\lambda)$ defined in \eqref{eq:sigma} satisfies the following statements:
  \begin{itemize}
      \item[a)] $\lam [ \chi  \mu + (1-\chi)  \nu] \le \sigma \le \lam \mu$, and hence $\sigma \ge \lambda \nu$;
      \item[b)] the function $\lam >0 \mapsto \sigma(\mu,\nu;\lambda)$ is non-decreasing;
      \item[c)] there holds
      \[
      \frac{1+\sigma}{\sigma} \le \min\left\{ \frac{1}{\chi}\left(1+\frac{1}{\lam \mu}\right),  1+\frac{1}{\lam \nu}\right\}.
      \]
  \end{itemize}
\end{lemma}

\begin{proof}
    a) It follows from \eqref{eq:sigma} and the fact that $\nu \le \mu$ that
 \[
    \sigma - \lam \mu =\frac{\lam(\mu-\nu)(1+\lam \mu)(\chi-1)}{1+\mu  \lambda+\chi  \lambda(\nu-\mu)}
    \ge \frac{\lam(\mu-\nu)(1+\lam \mu)(\chi-1)}{1+\mu  \lambda} =
\lam(\mu-\nu)(\chi-1),
    \]
    and hence that $\sigma \ge \lam [ \chi  \mu + (1-\chi)  \nu]$. Moreover, the second inequality of (a) immediately follows from the first identity above and the facts that $\chi<1$ and $\nu \le \mu$.

    b) Viewing $\sigma=\sigma(\lam)$ in \eqref{eq:sigma} as a function of $\lam$, one can verify that
\[
\sigma(\lam) = \frac{\nu(1+\mu \lam)+\chi(\mu-\nu)}{\lam^{-1}+\mu +\chi (\nu-\mu)}
\]
is a non-decreasing function of $\lam$.

    c) It follows from the definition of $\sigma$ in \eqref{eq:sigma} that 
    \[
    \frac{1+\sigma}{\sigma} = \frac{(1+\nu {\lam})(1+\mu {\lam})}{{\lam}[\nu(1+ \mu {\lam})+\chi(\mu-\nu)]}.
    \]
    Hence, it suffices to prove that
\begin{equation}\label{ineq:combine}
    \frac{(1+\nu {\lam})(1+\mu {\lam})}{{\lam}[\nu(1+ \mu {\lam})+\chi(\mu-\nu)]} \le \min\left\{ \frac{1}{\chi}\left(1+\frac{1}{{\lam} \mu}\right),  1+\frac{1}{{\lam} \nu}\right\}. 
\end{equation}
Since $\chi \in [0,1)$, it is easy to verify that
\[
\frac{1+\nu {\lam}}{\nu(1+ \mu {\lam})+\chi(\mu-\nu)} \le \frac{1}{\chi \mu},
\]
and hence that
\[
\frac{(1+\nu {\lam})(1+\mu {\lam})}{{\lam}[\nu(1+ \mu {\lam})+\chi(\mu-\nu)]} \le \frac{1+ \mu {\lam}}{\chi {\lam} \mu} = \frac{1}{\chi}\left(1+\frac{1}{{\lam} \mu}\right).
\]
Moreover, noting that $\chi \in [0,1)$ and $\mu\ge \nu$, we also have 
\[
\frac{(1+\nu {\lam})(1+\mu {\lam})}{{\lam}[\nu(1+ \mu {\lam})+\chi(\mu-\nu)]} \le 1+\frac{1}{{\lam} \nu}.
\]
Combining the above two inequalities, we conclude that \eqref{ineq:combine} holds.
\end{proof}

\subsection{Proofs of Lemmas  \ref{lem:tj} and \ref{lem:tj2}}\label{appendixupb}

Before presenting the proof of Lemma \ref{lem:tj},
we present and prove three technical results.

The  first one summarizes some basic properties of
U-PB.

\begin{lemma}\label{lem:101}
Assume that $\lam$ is the prox stepsize 
and $\hat x_{k-1}$ is the prox-center of a cycle of U-PB. Then, the following statements 
about this cycle hold:
	    \begin{itemize}
	        \item[a)] for every null iteration $j$ of the cycle, there exists a
         function $\overline f_{j}(\cdot)$ such that
	        \begin{align}
	            &\tau (\overline f_{j}+h) + (1-\tau) [\ell_f(\cdot;x_j)+h] \le f_{j+1} +h\le \phi, \label{eq:Gamma_j} \\
	            &\overline f_{j}+h \in \nConv{n}, \quad \overline f_{j}(x_j) = f_{j}(x_j),\label{eq:relation1}\\  
	       & x_j = \underset{u\in \R^n}\argmin \left \{\overline f_{j}(u)+h(u) + \frac{1}{2\lam} \|u-\hat x_{k-1}\|^2 \right\}, \label{eq:relation}
	        \end{align} 
         where 
         \begin{equation}\label{deftauk}
            \tau :=\frac{u_{\lambda}}{1+u_{\lambda}}
        \end{equation} 
         and $u_\lam$ is as in \eqref{def:ulam};
	        \item[b)] for every iteration $j$ of the cycle and $u\in \R^n$, we have
       \begin{equation}\label{ineq:Gammaj}
           \overline f_{j}(u)+h(u) + \frac1{2\lam}\|u-\hat x_{k-1}\|^2\ge m_j+ \frac1{2 \lam}\|u-x_j\|^2 ,
       \end{equation}
       where 
       \begin{equation}\label{defmj}
            m_j:=(f_j+h)(x_j)+\frac{1}{2 \lambda}\|x_j-\hat x_{k-1}\|^2.
        \end{equation}
	    \end{itemize}
	\end{lemma}
\begin{proof} a) Since $j$ is a null iteration of the $k$-th serious cycle, it follows from step 3.a of U-PB that
        $f_{j+1}=\mbox{BU}(\hat x_{k-1},x_j,f_j,\lambda)$.
        Using the properties of the BU blackbox, this implies the existence of $\overline f_j$ satisfying
\eqref{eq:relation1}, \eqref{eq:relation} and
\begin{equation}\label{eqfjp1}
        \max\{{\bar f}_j+h, \ell_f(\cdot;x_j)+h\} \le f_{j+1}+h \le \phi.
        \end{equation}
The statement now follows from the fact that \eqref{eqfjp1}  immediately implies \eqref{eq:Gamma_j}.

	    b) Since the objective function in the last identity of \eqref{eq:relation} is $\lam^{-1}$-strongly convex, we have
     	    \begin{equation}\label{ineafbar}
	    \overline f_{j}(u) + h(u)+\frac{1}{2\lam} \|u-\hat x_{k-1}\|^2  \ge \overline f_{j}(x_j) +h(x_j)+ \frac{1}{2\lam} \|x_j-\hat x_{k-1}\|^2  + \frac{1}{2 \lam}\|u-x_j\|^2,
	    \end{equation}
	   which immediately implies \eqref{ineq:Gammaj} in view of 
       the  identity
       in \eqref{eq:relation1} and the definition of $m_j$ in \eqref{defmj}.
\end{proof}

The second technical result provides a recursive bound on the quantity $m_j$
defined in \eqref{defmj}.

\begin{lemma}\label{lem:mj}
 If $j$ is a null iteration of a cycle, then
\begin{align}\label{ineqmtauok}
m_{j+1} - \tau m_j  \geq (1-\tau)   \left[\phi(x_{j+1}) +
		\frac{1}{2\lam} \|x_{j+1}-\hat x_{k-1}\|^2\right]-  \frac{(1-\tau) \varepsilon}{2}
\end{align}
where $\tau$ is defined in \eqref{deftauk} and $\varepsilon = (1-\chi)\bar \varepsilon/2$. 
\end{lemma} 

\begin{proof}
    Using the definitions of $\tau$ and $u_\lam$  in \eqref{deftauk} and \eqref{def:ulam}, respectively, we have  
\begin{equation}\label{reltauk}
\frac{\tau}{2\lam(1-\tau)} \stackrel{\eqref{deftauk}} = \frac{u_\lam}{2\lam} \stackrel{\eqref{def:ulam}}= 2 L_f + \frac{4M_f^2}{(1-\chi)\bar \varepsilon} = 2 L_f + \frac{2  M_f^2}{\varepsilon} \ge  \frac{  L_f}2 + \frac{2  M_f^2}{\varepsilon},
\end{equation}
where the last identity is due to the relation $\varepsilon = (1-\chi)\bar \varepsilon/2$.
Using the definition of $m_j$ in \eqref{defmj}, and relations \eqref{eq:Gamma_j} and \eqref{ineq:Gammaj} with $u=x_{j+1}$, we have
		\begin{align*}
		m_{j+1} & \overset{\eqref{defmj}}{=} (f_{j+1}+h)(x_{j+1}) + \frac{1}{2\lam} \|x_{j+1}-\hat x_{k-1}\|^2\\
		&\overset{\eqref{eq:Gamma_j}}{\ge} (1-\tau) \left[\ell_f(x_{j+1};x_j) + h(x_{j+1}) + \frac{1}{2\lam} \|x_{j+1}-\hat x_{k-1}\|^2\right] + \tau \left( (\overline f_{j}+h)(x_{j+1}) + \frac{1}{2\lam} \|x_{j+1}- \hat x_{k-1}\|^2 \right) \\
		& \overset{\eqref{ineq:Gammaj}}{\ge} (1-\tau) \left[\ell_f(x_{j+1};x_j) + h(x_{j+1}) + \frac{1}{2\lam} \|x_{j+1}-\hat x_{k-1}\|^2\right] + \tau \left( m_j + \frac{1}{2 \lam} \|x_{j+1} -x_j\|^2 \right).
		\end{align*}
  This inequality and \eqref{reltauk} then imply 
  \begin{equation}\label{ineq:mj}
      m_{j+1} - \tau m_j \ge (1-\tau) \left[ \ell_f(x_{j+1};x_j) + h(x_{j+1}) + \frac{1}{2\lam} \|x_{j+1}-\hat x_{k-1}\|^2 +  \left(\frac{  L_f}2 +\frac{2  M_f^2}{\varepsilon}\right) \|x_{j+1}-x_j\|^2 \right].
  \end{equation}
Using \eqref{ineq:est} with $(x,y)=(x_{j+1},x_j)$ and the fact that $\phi=f+h$, we have
        \begin{equation}\label{ineq:ellphi}
            \ell_f(x_{j+1};x_j) + h(x_{j+1}) + \frac{  L_f}2 \|x_{j+1}-x_j\|^2\ge \phi(x_{j+1}) - 2   M_f \|x_{j+1}-x_j\|.
        \end{equation}
        Using this inequality and relation \eqref{ineq:mj}, we conclude that
        \begin{align*}
&m_{j+1} - \tau m_j \\
		&\overset{\eqref{ineq:mj},\eqref{ineq:ellphi}}{\ge} (1-\tau) \Big(\phi(x_{j+1}) + \frac{1}{2 \lambda}\|x_{j+1}-\hat x_{k-1}\|^2\Big) +\frac{2(1-\tau)}{ \varepsilon}\Big(  M_f^2\|x_{j+1} -x_j\|^2 - M_f \varepsilon\|x_{j+1} -x_j\|\Big) \\
	  &\hspace*{0.4cm} \ge   (1-\tau)\left[\phi(x_{j+1}) +
		\frac{1}{2\lam} \|x_{j+1}-\hat x_{k-1}\|^2\right]  -  \frac{(1-\tau) \varepsilon}{2}, 
  \end{align*}
		where the last inequality follows  the inequality $ a^2-2ab \ge - b^2$ with $a=M_f\|x_{j+1}-x_j\|$ and $b=\varepsilon/2$.
\end{proof}

The next technical result establishes an important recursive formula for the sequence $\{t_j\}$ defined in \eqref{ineq:hpe1}.

\begin{lemma} \label{lembdj}
Let $j$ be an iteration of a cycle with step size $\lambda$
and first iteration $i$.
Then 
\begin{equation}\label{eqtjit}
t_j-\frac{(1-\chi)\bar \varepsilon}{4} 
\leq \tau^{j-i} \left(t_i-\frac{(1-\chi)\bar \varepsilon}{4}\right)
\end{equation}
where $\tau$ is defined in \eqref{deftauk}. 
\end{lemma}
\begin{proof} If $j=i$, \eqref{eqtjit} trivially holds.
We next show that \eqref{eqtjit} holds for  $j \geq i+1$. It suffices to show that for any 
$\ell \in \{i,\ldots,j-1\}$,
\begin{equation}\label{ineq:tj-recur}
         t_{\ell+1}-\frac{(1-\chi)\bar \varepsilon}{4} \le \tau \left(t_{\ell} - \frac{(1-\chi)\bar \varepsilon}{4}  \right).
     \end{equation}
     Let $\ell \in \{i,\ldots,j-1\}$  be given and note that $\ell$ is a null iteration of the cycle. Using the definition of $t_\ell$  in \eqref{ineq:hpe1} and Lemma \ref{lem:mj}
     with $j=\ell$, we have
        \begin{align*}
            t_{\ell+1}- \tau t_{\ell}  \stackrel{\eqref{ineq:hpe1}}= & \bar \phi_{\ell+1}   - m_{\ell+1}- \tau [ \bar \phi_{\ell}   - m_{\ell} ] = [ \bar \phi_{\ell+1}   - \tau \bar \phi_{\ell} ] -  [ m_{\ell+1}- \tau  m_{\ell} ] \\   \stackrel{\eqref{ineqmtauok}}\leq  &\bar \phi_{\ell+1} -\tau \bar \phi_{\ell}- (1-\tau)\left[\phi(x_{\ell+1}) +
		\frac{1}{2\lam} \|x_{\ell+1}-\hat x_{k-1}\|^2 \right] + \frac{(1-\tau) (1-\chi){\bar \varepsilon}}{4} \leq \frac{(1-\tau)(1-\chi) {\bar \varepsilon}}{4}, 
        \end{align*}
where the last inequality follows from \eqref{def:txj2} and the fact that $\tau \in (0,1)$. We have thus shown that \eqref{ineq:tj-recur} holds for
any $\ell \in \{i,\ldots,j-1\}$.
\end{proof}

\vgap

We are now ready to give the proof of the first main result of this subsection.

\vgap
\noindent
 {\textbf{Proof of Lemma \ref{lem:tj}}} 
By \eqref{eqtjit}, we obtain
\begin{align}\label{ineqtepsok}
\frac{4 t_{j}} {(1-\chi)\bar \varepsilon} -1 \leq
\frac{4} {(1-\chi)\bar \varepsilon} \tau^{j-i}t_i.
\end{align}
Using relation \eqref{deftauk}, we then have
     \[
     0 < \log \left( \frac{4 t_{j}} {(1-\chi)\bar \varepsilon} - 1 \right)  \stackrel{\eqref{ineqtepsok}}{\le} \log\left( \frac{4 \tau^{j-i} t_i}{(1-\chi)\bar \varepsilon} \right) \stackrel{\eqref{deftauk}}= \log\left( \frac{4  t_i}{(1-\chi)\bar \varepsilon} \right) + \log \left( \left[\frac{u_{\lambda}}{1+u_{\lambda}} \right]^{j-i} \right) \le \log\left( \frac{4  t_i}{(1-\chi)\bar \varepsilon} \right) -
     \frac{j-i}{1+u_{\lambda}}
     \]
     where for the last inequality we have used
     the $\log(x)\leq x-1$ which holds for $x>0$. This completes the proof of 
     \eqref{ineq:tj-recur2}.
\QEDA

\vgap

We next prove the second main result of this subsection.

\vgap

\noindent
{\textbf{Proof of Lemma \ref{lem:tj2}.}}
We first show that the inequality
\begin{equation}\label{ineqcrucialfirst}
t_i \leq  2  M_f \|x_i-{\hat x}_{k-1}\|
 + \frac{  L_f}{2}\|x_i-{\hat x}_{k-1}\|^2 - 
 \frac{1-\chi}{2\lam} \|x_i-{\hat x}_{k-1}\|^2 
\end{equation}
holds without making the assumptions of either case a) or case b).
Indeed,
using the definition of $t_j$  in \eqref{ineq:hpe1}
for $j=i$
and relation \eqref{def:txj2}, we have
\begin{align*}
t_i & \stackrel{\eqref{defmj}}{=} \bar \phi_i - (f_i + h)(x_i)-\frac1{2\lambda}\|x_i-{\hat x}_{k-1}\|^2 \\
& \stackrel{\eqref{def:txj2}}\leq \phi(x_{i}) +
		\frac{\chi}{2\lam} \|x_{i}-\hat x_{k-1}\|^2 -
(f_i + h)(x_i) - \frac1{2 \lambda}\|x_i-{\hat x}_{k-1}\|^2\\
& \;=  f(x_i)-
f_i(x_i)+\frac{\chi-1}{2 \lambda}\|x_i-{\hat x}_{k-1}\|^2,
\end{align*}
where the last equality follows from the fact that $\phi=f+h$.
Since $i$ is the first iteration of the cycle, it follows from step 3.d of U-PB that
$f_i(\cdot) \geq \ell_f(\cdot;{\hat x}_{k-1})$.
Combining this inequality with the above one, and using \eqref{ineq:est} with 
$(x,y)=(x_i,\hat x_{k-1})$, we conclude that
\begin{align}
t_i & \le 
f(x_i)-
\ell_{f}(x_i;{\hat x}_{k-1})+\frac{\chi-1}{2 \lambda}\|x_i-{\hat x}_{k-1}\|^2 \nn \\
 &
\stackrel{\eqref{ineq:est}}{\le}
 2  M_f \|x_i-\hat x_{k-1}\| + \frac{  L_f}2 \|x_i-{\hat x}_{k-1}\|^2+\frac{\chi-1}{2 \lambda   } \|x_i-{\hat x}_{k-1}\|^2 \label{ineq:reuse} \\
 &= 2  M_f \|x_i-{\hat x}_{k-1}\| - 
 \frac{1-\chi-\lam   L_f}{2\lam} \|x_i-{\hat x}_{k-1}\|^2 \nn
\end{align}
and hence that \eqref{ineqcrucialfirst} holds.

We next prove that a)  holds. Assume that $\lam>0$ is such that $\lam \le (1-\chi)/(2L_f)$. Note that this implies that 
 $1-\chi \ge 2\lambda L_f >0$, and hence
 $1-\chi - \lam L_f>0$.
Maximizing the right-hand side of \eqref{ineqcrucialfirst}  with respect to $\|x_i-\hat x_{k-1}\|$,  we conclude that
$t_i \leq (2 \lambda {  M}_f^2)/(1-\chi-\lambda {  L}_f)$, and hence that
\eqref{ineq:ti} due to the inequality $(1-\chi)/2 \ge \lam L_f$.

We next show that b) holds. Assume that (A4) holds.
Using \eqref{ineqcrucialfirst}, we obtain
 \begin{align*}
t_i &  \leq 2  M_f \|x_i-{\hat x}_{k-1}\|
 + \frac{  L_f D}{2}\|x_i-{\hat x}_{k-1}\| - 
 \frac{1-\chi}{2\lam} \|x_i-{\hat x}_{k-1}\|^2 \\
 & = \left( 2  M_f + \frac{  L_f D}{2} \right) \|x_i-{\hat x}_{k-1}\|
- 
 \frac{1-\chi}{2\lam} \|x_i-{\hat x}_{k-1}\|^2 \\
 &\le \left( 2  M_f + \frac{  L_f D}{2} \right)^2 \frac{\lam}{2(1-\chi)}
 \leq \frac{\lam(4 M_f^2 + L_f^2 D^2/4)}{1-\chi},
\end{align*}
where the second last inequality is due to the inequality that
$-a_0 x^2+b_0 x \leq b_0^2/(4a_0)$ for $a_0>0$, and the last inequality is due to the fact that $(a+b)^2 \le 2(a^2+b^2)$.
\QEDA















\if{
Assume 
\[
\lam \le 
 (1-\chi) \lam_0
 \]
 Need to to enforce 
 \[
 \lam \alpha \tau^{N-1} \le \varepsilon
 \]
 where
 \[
 \alpha = \frac{4M^2}{1-\chi}
 \]
 A sufficient condition is
 \[
(1-\chi)\lam_0 \alpha \tau^{N-1} \le \varepsilon
 \]
 or
 \[
 \theta \tau^{N-1} \le \varepsilon
 \]
 or
 \[
 \frac{\theta}{\varepsilon} \le \left( \frac{1}{\tau} \right)^{N-1}
 \]
 where
 \[
 \theta =(1-\chi) \lam_0 \alpha = 4 M^2 \lam_0
 \]
 Taking log of both sides, we  get
 \[
 \log 
 \frac{\theta}{\varepsilon} \le (N-1) \log \left( \frac{1}{\tau} \right)
 \]
 A sufficient condition for the above to hold is
 \[
 \log 
 \frac{\theta}{\varepsilon} \le (N-1) \frac{\tau^{-1}- 1}{\tau^{-1}} = (N-1)(1-\tau)
 \]
 Now 
 \[
 \tau = 1 - (1+ \lam \beta)^{-1}
 \]
 where $\beta =4T^2/\varepsilon$. So
 \[
 (1+\lam \beta)^{-1} \ge \frac{1}{N-1} \log 
 \frac{\theta}{\varepsilon} 
 \]
 or
 \[
 1+\lam \beta \le \frac{N-1}{\log (\theta\varepsilon^{-1})}
 \]
 or
 \[
 \lam \le \frac1\beta \left[ 
 \frac{N-1}{ \log (\theta\varepsilon^{-1})} -1 \right] 
 = \frac{\varepsilon}{4T^2} \left[ 
 \frac{N-1}{\log (\theta\varepsilon^{-1})}
-1 \right] 
 \]
 Assume
 $ \theta \varepsilon^{-1}=aN$.
 \[
 \frac{N-1}{\log (aN)}-1 \ge \frac{N-1}{2\log (a N)} = \frac{N-1}{2\log a + 2 \log N} = \frac{N-1}{ 2 \log N - 1}
 \ge \frac{N-1}{ 2 \log N }
 \]
 or
 \[
 \frac{N-1}{2\log (aN) } \ge 1
 \]
 For $N=2$, it means
 \[
 \frac{1}{2 \log(2a)} =1
 \]
 or
 \[
 \log 2a = \frac12
 \]
 or
 \[
 \log a = -\frac12
 \]
 Now, if $N \ge 2$, choose $\lam_0$ so that
 \[
 \log (\theta \varepsilon^{-1} ) = \log( 4 M^2 \lam_0 \varepsilon^{-1} ) = 1/2
 \]
 Then
 \[
 \lam \le \frac{\varepsilon}{4T^2} \left[ 
 \frac{N-1}{\log (\theta\varepsilon^{-1})}
-1 \right] = \frac{(2N-3) \varepsilon}{4T^2}

Try
\[
\log(\theta \varepsilon^{-1})  = \log N
\]
or
\[
4 M^2 \lambda_0 \varepsilon^{-1}= \theta \varepsilon^{-1} =N
\]
So
\[
\lam_0 = \frac{N \varepsilon}{4M^2}
\]
Try
\[
\log(\theta \varepsilon^{-1})  =
\log \log N
\]
or
\[
4 M^2 \lambda_0 \varepsilon^{-1}= \theta \varepsilon^{-1} =\log N
\]

So
\[
\lam_0 = \frac{\varepsilon\log N }{4M^2}
\]

]

$$
\lambda \leq (1-\chi)\lambda_0 =(1-\chi) \frac{\varepsilon\log N }{4M^2}
$$

}\fi

\end{document}